\documentclass[a4paper,11pt,reqno]{article}

\usepackage[utf8]{inputenc}
\usepackage[T1]{fontenc}
\usepackage{lmodern}
\usepackage[english]{babel}
\usepackage{microtype,csquotes}
\usepackage[title]{appendix}
\usepackage{multirow}

\usepackage[backend=biber,hyperref=true,style=alphabetic,giveninits=true,maxnames=50,sorting=nyt]{biblatex}
\usepackage{tikz}
\usetikzlibrary{math}

\definecolor{midnightblue}{rgb}{0.1, 0.1, 0.44}
\usepackage[colorlinks,allcolors=midnightblue,pdfusetitle]{hyperref}
\AtBeginDocument{}
\AtBeginDocument{}
\AtBeginDocument{}
\hypersetup{
	pdfsubject={Concentration of the Born-Jordan distribution},
	pdfkeywords={time-frequency analysis,Born--Jordan distribution,phase-space concentration,critical exponent}}

\usepackage[a4paper,hmargin=2.5cm,bottom=2.5cm,top=2.5cm,footskip=3\baselineskip,
	marginparwidth=2.4cm,marginparsep=0.05cm]{geometry}

\usepackage{enumitem}
\setlist[enumerate,1]{label=(\roman*)}

\usepackage{booktabs}
\newcommand{\otoprule}{\midrule[\heavyrulewidth]}

\usepackage{fsmath}
\usepackage{bm}

\newcommand{\BJ}{\mathrm{BJ}}
\DeclareMathOperator{\Op}{Op}

\newcommand{\zjn}{z_j^{(n)}}
\newcommand{\zjpn}{z_{j'}^{(n)}}

\newcommand{\wkn}{w_k^{(n)}}

\newcommand{\Fkn}{F_k^{(n)}}

\title{\texorpdfstring{\vspace{-2cm}}{}On the existence of optimizers for nonlinear time-frequency concentration problems: the Born--Jordan distribution}

\author{Federico Stra\thanks{Politecnico di Torino, \url{federico.stra@polito.it}}
	\and Erling A. T. Svela\thanks{Norwegian University of Science and Technology, \url{erling.a.t.svela@ntnu.no}}
	\and S. Ivan Trapasso\thanks{Politecnico di Torino, \url{salvatoreivan.trapasso@polito.it}}}
\date{}

\begin{document}

\maketitle

\thispagestyle{empty}
\enlargethispage{1.5\baselineskip}

\vspace{-1.5em}
\begin{abstract}
	We study the $L^p$ concentration problem for the Born--Jordan distribution in dimension \mbox{$d>1$},
	thus extending the one-dimensional analysis in [Stra--Svela--Trapasso, \emph{J.\ Math.\ Pures Appl.}\ (2026)].
	We show that the existence of concentration optimizers depends on the exponent $p$ with a
	critical threshold at $p_*(d)= \frac{2d}{d-2}$ for $d\geq2$ (with the understanding that
	$p_*(2)=\infty$). In particular, for subcritical exponents $1\leq p<p_*(d)$ we prove that the
	supremum is finite and is attained, whereas for supercritical exponents $p>p_*(d)$ we show that
	the functional is unbounded. We also provide the complete solution in the (significantly more)
	challenging critical regime in dimension $d=2$. For the critical case in dimension $d=2$ and the
	supercritical case in dimension $d>2$ we also show that $L^2$ functions whose Born--Jordan
	distribution does not belong to $L^p$ are generic in the Baire sense. For the critical exponent in
	dimension $d>2$, for which Wang has recently shown that the supremum is finite, we prove that the
	functional is not sequentially weakly upper semicontinuous, showing that the still open question
	of attainment is nontrivial.
\end{abstract}

\vspace{-1em}

\tableofcontents
\vspace*{1.5\baselineskip}

\section{Introduction}

The Wigner distribution, defined for \(f\in L^2(\setR^d)\) by
\begin{align*}
	W f(z)
	  & \coloneqq \int_{\setR^d} e^{-2\pi i\xi\cdot y} f(x+y/2) \overline{f(x-y/2)} \d y, &
	z & = (x,\xi) \in \setR^{2d},
\end{align*}
is one of the most popular quadratic phase-space representations for signals in view of its many
desirable properties --- including being real-valued, continuous and energy preserving. However, as
no time-frequency representation may possess every desirable property at once~\cite{Janssen97}, even
the Wigner transform comes with some features that may represent challenging drawbacks in several
applications. For instance, although it may be informally interpreted as a phase-space energy
density of the signal $f$, $Wf$ will in general take both positive and negative values. Moreover,
the inherent sesquilinear structure of the Wigner transform produces peculiar \textit{ghost interferences}.
To be more precise, this means that the magnitude of the cross-term \(2\Re W(f,g)\) in the expansion
of \(W(f+g)\) may be significantly larger than that of the main terms
\(Wf+Wg\)~\cite{Boggiatto-deDonno-Oliaro-2010}. Such scattered interference patterns further
obstruct the interpretation of $Wf$ as a genuine energy density, hence several alternative phase
space distributions have been introduced in order to tame or circumvent these phenomena. In this
connection, the \textit{Cohen class}~\cite{Cohen66} consists of distributions of the form
$Qf = Wf*\sigma$, obtained by smoothing the Wigner transform with a suitable convolution kernel.

A distinguished member of this family is the \emph{Born--Jordan distribution}~\cite{deGosson16,HlawatschAuger}, originally arising in the context of quantization problems in~\cite{BJ1925} and popularized as an average of certain generalized Wigner distributions that is effectively able to reduce the magnitude of the ghost frequencies~\cite{Boggiatto-deDonno-Oliaro-2010,CdGN18}.

\begin{definition}\label{def:Wigner Born-Jordan}
	Let $f\in L^2(\setR^d)$.
	For $\tau\in(0,1)$, define the \emph{$\tau$-Wigner distribution}
	\begin{align}\label{eq:W-tau}
		W_\tau f(z)
		  & \coloneqq \int_{\setR^d} e^{-2\pi i\xi\cdot y} f\bigl(x+\tau y\bigr) \overline{f\bigl(x-(1-\tau)y\bigr)} \d y, &
		z & = (x,\xi) \in \setR^{2d}.
	\end{align}
	The \emph{Born--Jordan distribution} of $f$ is defined by%
	\footnote{See the end of \autoref{sec:notation} for some details about how this definition is to be understood.}
	\begin{equation}\label{eq:BJ}
		W_\BJ f(z) \coloneqq  \int_0^1 W_\tau f(z) \d\tau.
	\end{equation}
\end{definition}

As a member of the Cohen family, the Born--Jordan distribution has a kernel, which was explicitly determined in~\cite{CdGN17}:
\[
	W_\BJ f=Wf*\Theta_\sigma,\qquad \Theta_\sigma(x,\xi)=\begin{cases}
		-2\mathrm{Ci}(4\pi |x\cdot\xi|)                          & (d=1)  \\
		\mathcal{F}(\bm{1}_{\{|s|\geq 2\}}|s|^{d-2})(x\cdot \xi) & (d>1),
	\end{cases}
\]
where $\mathrm{Ci}(t)=-\int_t^{+\infty} \frac{\cos s}{s} \d s$ is the cosine integral function. This peculiar kernel is the root of the poor behavior of the Born--Jordan distribution regarding regularity and decay, since \(\Theta_\sigma\) is not globally integrable in any \(p\)-norm, and not even locally \(L^\infty\)~\cite{CdGN17}. In fact, many of the properties of the Born--Jordan distribution, such as continuity and boundedness, are only known to hold when \(d=1\) due to the decay of the kernel degrading as \(d\) grows.%
\footnote{As we will see in \autoref{thm:attainment-baire}, they fail in general for $d\geq2$, but a local form of continuity survives for annular supports, cf.~\autoref{sec:incomplete}.}

In spite of these considerations, the empirical improvement in terms of interference reduction of the Born--Jordan distribution compared to the Wigner one has been recently confirmed in the context of phase-space concentration problems, a classical topic in time-frequency analysis. To be precise, in this article we study the maximal $L^p$ concentration of the Born--Jordan distribution on a
prescribed phase space subset $\Omega\subset \setR^{2d}$ with positive finite Lebesgue measure, and in particular whether the optimal concentration
\begin{equation}\label{eq:concentration problem}
	\sup_{\substack{f\in L^2(\setR^d) \\ \norm{f}_{L^2}^2=1}} \norm{W_\BJ f}_{L^p(\Omega)}
	= \sup_{f\in L^2(\setR^d)\setminus\{0\}} \frac{\norm{W_\BJ f}_{L^p(\Omega)}}{\norm{f}_{L^2}^2}
\end{equation}
is finite and attained. The existence of \textit{local} concentration optimizers has been
investigated only recently in~\cite{Nicola-Romero-Trapasso-2022} for the ambiguity transform and
in~\cite{Stra-Svela-Trapasso-2025} by the authors in the case of the Wigner transform in the
framework of concentration compactness at the phase space level. As a byproduct, in that work we
were able to address the Born--Jordan concentration problem in the $1$-dimensional case, as
summarized below.

\begin{theorem}[$d=1$, \cite{Stra-Svela-Trapasso-2025}]\label{thm:BJ d=1}
	Let $\Omega\subset \setR\times \setR$ be a measurable set with $0<\leb^{2}(\Omega)<\infty$.
	Then for every $p\in[1,\infty)$ the functional
	\[
		J_\BJ^p: L^2(\setR) \to [0,\infty): f\mapsto\norm{W_\BJ f}_{L^p(\Omega)}
	\]
	is sequentially weakly continuous in $L^2(\setR)$. As such, the supremum
	\[
		\sup_{f\in L^2(\setR)\setminus\{0\}} \frac{\norm{W_\BJ f}_{L^p(\Omega)}}{\norm{f}_{L^2(\setR)}^2}
		< \infty
	\]
	is attained by an optimizer. When $p=\infty$ we have
	\[
		\sup_{f\in L^2(\setR)\setminus\{0\}} \frac{\norm{W_\BJ f}_{L^\infty(\Omega)}}{\norm{f}_{L^2(\setR)}^2}
		= \pi,
	\]
	but the supremum is not attained.
\end{theorem}

These findings are particularly interesting as they show that, in spite of the challenging nature of the Born--Jordan kernel already discussed, the underlying interference reduction mechanism occurs already at the topological level --- it is the first non-trivial example of a time-frequency concentration functional that is proved to be weakly continuous. The goal of the present article is to provide a fairly complete picture of the matter in higher dimensions, solving a problem raised at the end of~\cite{Stra-Svela-Trapasso-2025}.

\subsection{Main results}

We show that the answer to the $L^p$ concentration problem for the Born--Jordan distribution changes
drastically depending on the exponent $p$ and we must therefore distinguish three regimes, where a
crucial role is played by the following \textit{critical exponent}.

\begin{definition}\label{def:critical exponent}
	For every positive integer $d\in\setN_+$ let $p_*(d)$ be the \emph{critical exponent} defined as
	\begin{align}\label{eq:p*}
		p_*(1) & = p_*(2) = \infty,                                                           &
		p_*(d) & = \left(\frac12-\frac1d\right)^{-1} = \frac{2d}{d-2} \qquad \forall\ d\ge 3.
	\end{align}
\end{definition}

We stress that the critical exponent for $d>2$ coincides with the exponent $2^*$ of the Sobolev embedding $W^{1,2}(\setR^d)\to L^{p_*(d)}(\setR^d)$, and this occurrence is not coincidental. Indeed, in (symplectic) Fourier variables the Born--Jordan distribution is obtained from the Wigner distribution by multiplication with $\Theta(x,\xi)=\operatorname{sinc}(\pi\,x\cdot\xi)$, and the asymptotic behavior $\abs{\operatorname{sinc}(t)} \lesssim |t|^{-1}$ as $\abs{t}\to+\infty$ suggests heuristically that the Cohen smoothing acts like an inverse first-order multiplier, so that $W_{\mathrm{BJ}}f$ gains roughly one derivative compared with $Wf$. In the natural $L^2$-Sobolev scale this is modeled by the Riesz potential $I=(-\Delta)^{-1/2}$, which is bounded $L^2(\setR^d) \to L^p(\setR^d)$ precisely at the critical exponent $p=p_*(d)$ for $d>2$. An equivalent viewpoint that will serve as a main motivation for several constructions in our proofs stems from scale invariance: under the $L^2$-preserving scaling $F_\lambda(x)=\lambda^{d/2}F(\lambda x)$ one has $\|IF_\lambda\|_{L^p} = \lambda^{d/2-1-d/p}\|IF\|_{L^p}$, so scale invariance occurs exactly at $p_*(d)$. In dimension $d=2$ the critical exponent escapes to $p_*(2)=\infty$, which is consistent with the fact that the only borderline case is the endpoint $p=\infty$.

The three following theorems are our main results and they settle the concentration
problem~\eqref{eq:concentration problem} in high dimension, with the exception of the critical case
in dimension greater than~$2$, which is discussed further on.

\begin{theorem}[Subcritical case in $d\geq2$]\label{thm:subcritical}
	Let $d\geq2$ be an integer and $\Omega\subset\setR^d\times\setR^d$ be a measurable set with
	$0<\leb^{2d}(\Omega)<\infty$. Then for every $p\in\bigl[1,p_*(d)\bigr)$ the functional
	\[
		J_\BJ^p: L^2(\setR^d) \to [0,\infty): f\mapsto\norm{W_\BJ f}_{L^p(\Omega)}
	\]
	is sequentially weakly continuous in $L^2(\setR^d)$.
	As a consequence,
	\[
		\sup_{f\in L^2(\setR^d)\setminus\{0\}} \frac{\norm{W_\BJ f}_{L^p(\Omega)}}{\norm{f}_{L^2(\setR^d)}^2}
		< \infty
	\]
	is attained by an optimizer.
\end{theorem}

\begin{theorem}[Critical case in $d=2$]\label{thm:critical d=2}
	Let $\Omega\subseteq\setR^2\times\setR^2$ be a measurable set with $\leb^4(\Omega)>0$. Then
	\[
		\sup_{f\in L^2(\setR^2)\setminus\{0\}} \frac{\norm{W_\BJ f}_{L^\infty(\Omega)}}{\norm{f}_{L^2(\setR^2)}^2}
		= \infty.
	\]
\end{theorem}

\begin{theorem}[Supercritical case in $d>2$]\label{thm:supercritical d>2}
	Let $d>2$ be an integer and $\Omega\subseteq\setR^d\times\setR^d$ be a measurable set with
	$\leb^{2d}(\Omega)>0$. Then for every $p\in\bigl(p_*(d),\infty\bigr]$
	\[
		\sup_{f\in L^2(\setR^d)\setminus\{0\}} \frac{\norm{W_\BJ f}_{L^p(\Omega)}}{\norm{f}_{L^2(\setR^d)}^2}
		= \infty.
	\]
\end{theorem}

Let us emphasize that, regarding the critical and supercritical cases, the previous \autoref{thm:critical d=2} and
\autoref{thm:supercritical d>2} do not provide information about the attainment or non-attainment of the
supremum, hence leaving open the question whether there exists a function $f\in L^2(\setR^d)$ such that $W_\BJ f\not\in L^p(\Omega)$, $\norm{W_\BJ f}_{L^p(\Omega)}=\infty$. The picture is completed by the following companion result.

\begin{theorem}\label{thm:attainment-baire}
	Under the assumptions of \autoref{thm:critical d=2} or
	\autoref{thm:supercritical d>2}, the set of $f\in L^2(\setR^d)$ such that
	$W_\BJ f\notin L^p(\Omega)$ is a dense $G_\delta$ subset of $L^2(\setR^d)$.
	In particular, the suprema in \autoref{thm:critical d=2} and
	\autoref{thm:supercritical d>2} are attained.
\end{theorem}

Since countable intersections of dense $G_\delta$ sets are again dense $G_\delta$ sets,
a refined argument yields generic functions that work for all sets with nonempty interior up to negligible sets and all exponents in the relevant range.

\begin{corollary}\label{cor:generic}
	Let $d\geq2$, set $P_d\coloneqq\{\infty\}$ if $d=2$ and $P_d\coloneqq\bigl(p_*(d),\infty\bigr]$ if $d>2$,
	and let $\mathcal{U}_{2d}$ be the family of measurable subsets $U\subseteq\setR^{2d}$
	with nonempty interior up to a negligible set.%
	\footnote{In the sense that there is a negligible set $Z\subset\setR^{2d}$ such that $U\cup Z$ has nonempty interior.}
	Then the set of functions
	\[
		\set*{f\in L^2(\setR^d)}{W_\BJ f\notin L^p(U)\quad \forall p\in P_d,\ \forall U\in\mathcal{U}_{2d}}
	\]
	is a dense $G_\delta$ subset of $L^2(\setR^d)$. In particular, for a generic
	$f\in L^2(\setR^d)$ the distribution $W_\BJ f$ is essentially unbounded on every
	nonempty open subset of $\setR^{2d}$.
\end{corollary}

The only case that eluded our efforts so far is the critical threshold $p=p_*(d)$ in dimension $d>2$.
After submitting this paper for publication, this question has been partially addressed by Y.~Wang~\cite{Wang26}:
there, the supremum of the concentration quotient \eqref{eq:concentration problem} is proved to be finite,
using substantially different techniques.
The question of critical attainment remains open, seemingly requiring even more sophisticated tools.
In particular, in the following theorem we show that critical attainment cannot be deduced as was done for the subcritical regime in \autoref{sec:subcritical}, namely from the sequential weak upper semicontinuity of the functional. This does not rule out the
existence of critical optimizers, but it shows that the question is genuinely nontrivial
(and still open, to our knowledge).

\begin{theorem}[Critical case in $d>2$]\label{thm:critical d>2}
	Let $d>2$ be an integer and $\Omega\subseteq\setR^d\times\setR^d$ be a measurable set with
	$\leb^{2d}(\Omega)>0$. Then the critical functional
	\[
		J_\BJ^{p_*(d)}: L^2(\setR^d) \to [0,\infty): f\mapsto\norm{W_\BJ f}_{L^{p_*(d)}(\Omega)}
	\]
	is not sequentially weakly upper semicontinuous in $L^2(\setR^d)$.
\end{theorem}

For the benefit of the reader, a comprehensive view of the matter is given in \autoref{tab:summary}.

\begin{table}[ht]
	\caption{Summary of the available results for the concentration problem
		\eqref{eq:concentration problem} of the Born--Jordan distribution.
		The supercritical case does not appear in dimension $1$ and $2$ since $p_*(1)=p_*(2)=\infty$. \enquote{Attained} infinite supremum means that there exists $f\in L^2(\setR^d)$ with $W_\BJ f\notin L^p(\Omega)$.}
	\label{tab:summary}
	\centering
	\begin{tabular}[t]{cccc}
		\toprule
		                       & $1\leq p<p_*(d)$                                      & $p=p_*(d)$                                                & $p>p_*(d)$                                                     \\
		\otoprule
		\multirow{2}{*}{$d=1$} & weakly continuous, finite, attained                   & finite, unattained                                        & \multirow{2}{*}{---}                                           \\
		                       & \cite[Prop.~6.4, Theo.~6.5]{Stra-Svela-Trapasso-2025} & \cite[Prop.~6.8]{Stra-Svela-Trapasso-2025}                &                                                                \\
		\midrule
		\multirow{2}{*}{$d=2$} & weakly continuous, finite, attained                   & infinite, attained                                        & \multirow{2}{*}{---}                                           \\
		                       & \autoref{thm:subcritical}                             & \autoref{thm:critical d=2} and~\ref{thm:attainment-baire} &                                                                \\
		\midrule
		\multirow{2}{*}{$d>2$} & weakly continuous, finite, attained                   & finite, not weakly u.s.c.                                 & infinite, attained                                             \\
		                       & \autoref{thm:subcritical}                             & \cite{Wang26}, \autoref{thm:critical d>2}                 & \autoref{thm:supercritical d>2} and~\ref{thm:attainment-baire} \\
		\bottomrule
	\end{tabular}
\end{table}

The article is organized as follows. The proofs of \autoref{thm:subcritical}, \autoref{thm:critical d=2},
\autoref{thm:supercritical d>2}, and \autoref{thm:critical d>2} are located in \autoref{sec:subcritical},
\autoref{sec:critical}, \autoref{sec:supercritical}, and \autoref{sec:critical d>2} respectively.
In \autoref{sec:indirect-baire} we provide a proof of \autoref{thm:attainment-baire} and \autoref{cor:generic}
based on the Baire category theorem.
In \autoref{sec:indirect} we provide instead a proof of the weaker \autoref{thm:attainment} based on the
unboundedness stated in \autoref{thm:critical d=2} and \autoref{thm:supercritical d>2},
the closed graph theorem, and the Banach--Steinhaus theorem (uniform boundedness
principle). While these arguments are fairly elegant and have the advantage of encompassing
both the critical (when $d=2$) and supercritical (when $d>2$) cases at once, they are
inherently non-constructive. For this reason, in \autoref{sec:attainment-supercritical} we
provide an explicit construction of a function $f\in L^2(\setR^d)$ (in fact, a whole
one-parameter family) such that $W_\BJ f\not\in L^p(\Omega)$ when $d>2$ and $p>p_*(d)$,
whereas in \autoref{sec:attainment-critical} we provide a somewhat more involved
construction of a function $f\in L^2(\setR^2)$ satisfying $W_\BJ f\not\in L^\infty(\Omega)$.
We believe that, in spite of the overlapping scopes, these approaches provide valuable
insights on the Born--Jordan distribution. In particular, the arguments in
\autoref{sec:indirect} and \autoref{sec:indirect-baire} explain the ultimate reason why a function with
$W_\BJ f\notin L^p(\Omega)$ must exist, while the explicit constructions in
\autoref{sec:attainment-critical} and \autoref{sec:attainment-supercritical} precisely
leverage the structural weak points of the Born--Jordan distribution.

The manuscript also includes two appendices: in \autoref{sec:incomplete} we consider \textit{truncated} Born--Jordan distributions and related properties (including an alternative route to \autoref{thm:critical d=2} and \autoref{thm:supercritical d>2} for $p=\infty$), whereas in \autoref{sec:quantization} we demonstrate some consequences of our results and of \cite{Wang26} in the context of Born--Jordan pseudodifferential operators with symbols in Lebesgue spaces, including the compactness of Born--Jordan operators at the critical symbol exponent (\autoref{prop:bj-endpoint-compact}).

\section{Notation and known facts}
\label{sec:notation}

The Lebesgue measure in $\setR^d$ and the $d$-dimensional Hausdorff measure (in any ambient dimension)
are denoted by $\leb^d$ and $\haus^d$ respectively;
$\omega_{d-1}=\haus^{d-1}(S^{d-1})=2\pi^{d/2}/\Gamma(d/2)$ denotes the surface area of the
$(d-1)$\nobreakdash-sphere in $\setR^d$ with unit radius.

If there is a constant \(C>0\) such that \(a\leq Cb\) we write \(a\lesssim b\), and if \(C\) depends on a parameter \(\lambda\) we write \(a\lesssim_\lambda b\). We write $a \asymp b$ to mean that both $a \lesssim b$ and $b \lesssim a$ hold.

For \(x,\xi\in \setR^d\) we denote translation by \(x\) as \(T_x f(y)=f(y-x)\) and modulation by \(\xi\) by \(M_{\xi}f(y)=e^{2\pi i \xi\cdot y} f(y)\). For \(z=(x,\xi)\) we denote the time-frequency shift by \(z\) as \(\pi(z)f(y)=M_\xi T_x f(y)\). If a sequence \(\{f^{(n)}\}\) in \(L^2(\setR^d)\) converges weakly to \(f\) we write \(f^{(n)}\weakto f\).

For \(\tau\in (0,1)\) and \(f,f_1,f_2,g_1,g_2\in L^2(\setR^d)\) the \(\tau\)-Wigner distributions are known to be elements of \(L^2(\setR^{2d})\cap C_0(\setR^{2d})\), covariant under time-frequency shifts, \(W_\tau(\pi(z)f)=T_zW_\tau f\), and the associated cross-$\tau$-Wigner distributions $W_\tau(f,g)(z)=\int_{\setR^d} e^{-2\pi i\xi\cdot y} f\bigl(x+\tau y\bigr) \overline{g\bigl(x-(1-\tau)y\bigr)} \d y $ satisfy Moyal's identity:
\[
	\langle W_{\tau}(f_1,g_1),W_{\tau}(f_2,g_2)\rangle_{L^2(\setR^{2d})}=\langle f_1,f_2\rangle_{L^2(\setR^{d})}\langle g_2,g_1\rangle_{L^2(\setR^{d})}.
\]
We also have the symmetry relation \(\overline{W_{\tau}f}=W_{1-\tau}f\).

For $f\in L^2(\setR^d)$ the map $\tau\mapsto W_\tau f$ is continuous with values in $L^2(\setR ^{2d})$ and bounded by $\norm{f}_{L^2(\setR^d)}^2$, so $W_\BJ f$ is also defined as an $L^2(\setR^{2d})$-valued Bochner integral $\int_0^1 W_\tau f\d\tau$. This of course is consistent with the pointwise definition \eqref{eq:BJ}: for every compact $K\subseteq\setR^{2d}$, Fubini--Tonelli and Cauchy--Schwarz give $\int_K\int_0^1\abs{W_\tau f(z)}\d\tau\d z\leq\leb^{2d}(K)^{1/2}\norm{f}_{L^2(\setR^d)}^2$, hence the integral in \eqref{eq:BJ} converges absolutely for almost every $z$ and coincides almost everywhere with the $L^2$-valued one.

The Born--Jordan distribution is also known to be real-valued and covariant under time-frequency shifts.

\section{Proofs}

\begin{definition}\label{def:f_rR}
	For all radii $0<r_1\leq r_2$, define the radial function $f_{r_1,r_2}\in L^2(\setR^d)$ given by
	\[
		f_{r_1,r_2}(x) = \abs{x}^{-d/2} \bm1_{[r_1,r_2]}(\abs{x}).
	\]
	For convenience, define also $f_R=f_{1,R}$ for every radius $R\geq1$.
\end{definition}

For the Born--Jordan problem the relevance of the functions $f_{r_1,r_2}$ comes from the fact that
at the origin we have
\[
	W_{\BJ}f(0,0) = \int_0^1 W_\tau f(0,0) \d\tau
	= \int_0^1 \bigl(\tau(1-\tau)\bigr)^{-d/2} \angles[\big]{D_{-\frac{\tau}{1-\tau}}f,f} \d\tau,
\]
where $D_\lambda f(x)=|\lambda|^{d/2}f(\lambda x)$ denotes the dilation by $\lambda$ that is an
isometry in $L^2(\setR^d)$. The functions $f_{r_1,r_2}$ are truncations of the dilation-invariant
potential $\abs{x}^{-d/2}$ and indeed, as $R\to\infty$, $f_R$ will turn out to be an asymptotic
maximizer of the optimization problems in \autoref{thm:critical d=2} and
\autoref{thm:supercritical d>2}.
For the moment, just notice that its squared $L^2$ norm is
\begin{equation}\label{eq:f_R L^2 norm}
	\norm{f_R}_{L^2(\setR^d)}^2
	= \int_{\set{x\in\setR^d}{1\leq\abs{x}\leq R}} \abs{x}^{-d} \d x
	= \int_1^R r^{-d} \omega_{d-1}r^{d-1} \d r
	= \omega_{d-1} \log R.
\end{equation}

\subsection{Preliminary lemmas: Wigner and Born--Jordan estimates}

We now arrive at the crucial estimates necessary to prove \autoref{thm:critical d=2} and
\autoref{thm:supercritical d>2}: \autoref{lem:Re W_tau(x xi) lower bound} provides a lower bound for
$\Re W_\tau f_R(x,\xi)$ in terms of $W_\tau f_R(x,0)$, \autoref{lem:W_tau(x 0) lower bound} bounds
$W_\tau f_R(x,0)$ from below, and finally \autoref{lem:primal BJ lower bound} establishes a lower
bound for $\Re W_\BJ f_R(x,\xi)$ by combining the previous two lemmas and integrating in $\tau$. For
convenience of application in the proofs of the theorems, \autoref{cor:W_BJ lower bound}
reformulates these lower bounds in a summarized way.

\begin{lemma}\label{lem:Re W_tau(x xi) lower bound}
	Let $R\geq1$, $\tau\in(0,1)$ and $(x,\xi)\in\setR^d\times\setR^d$ such that
	\[
		4(R+\abs{x})\abs{\xi} \leq 1.
	\]
	Then
	\[
		\Re W_\tau f_R(x,\xi) \geq \cos\bigl(4\pi(R+\abs{x})\abs{\xi}\bigr) W_\tau f_R(x,0).
	\]
\end{lemma}

Naturally, the lemma becomes more useful when $8(R+\abs{x})\abs{\xi}<1$, so that
$\abs[\big]{4\pi(R+\abs{x})\abs{\xi}}<\pi/2$ and therefore
$\cos\bigl(4\pi(R+\abs{x})\abs{\xi}\bigr)>0$.

\begin{proof}
	The crucial observation is that in the definition of $W_\tau f_R(x,\xi)$ we can restrict the
	domain of integration to just
	\[
		\set{y\in\setR^d}{\abs{y}\leq2(R+\abs{x})}.
	\]
	Indeed, for the integrand function to be nonzero we must have both $1\leq\abs{x+\tau y}\leq R$ and
	$1\leq\abs{x-(1-\tau)y}\leq R$, from which in particular $\tau\abs{y}-\abs{x}\leq R$ and
	$(1-\tau)\abs{y}-\abs{x}\leq R$, hence
	\[
		\abs{y} \leq \min\{\tau^{-1},(1-\tau)^{-1}\} (R+\abs{x}) \leq 2(R+\abs{x}).
	\]
	Using the fact that $f_R\geq0$ and $\abs{2\pi\xi\cdot y}\leq2\pi\abs{\xi}2(R+\abs{x})\leq \pi$,
	we finally deduce
	\[
		\begin{split}
			\Re W_\tau f_R(x,\xi)
			 & = \int_{\{\abs{y}\leq2(R+\abs{x})\}}
			\cos(2\pi\xi\cdot y) f_R\bigl(x+\tau y\bigr) f_R\bigl(x-(1-\tau)y\bigr) \d y        \\
			 & \geq \cos\bigl(4\pi(R+\abs{x})\abs{\xi}\bigr) \int_{\{\abs{y}\leq2(R+\abs{x})\}}
			f_R\bigl(x+\tau y\bigr) f_R\bigl(x-(1-\tau)y\bigr) \d y                             \\
			 & = \cos\bigl(4\pi(R+\abs{x})\abs{\xi}\bigr) W_\tau f_R(x,0). \qedhere
		\end{split}
	\]
\end{proof}

\begin{lemma}\label{lem:W_tau(x 0) lower bound}
	Let $R\geq1$, $\tau\in(0,1/2]$, and $x\in\setR^d$ such that
	\[
		\frac{1+\abs{x}}\tau \leq R-\abs{x}.
	\]
	Then
	\[
		W_\tau f_R(x,0) \geq 2^{-d/2}\omega_{d-1} \tau^{-d/2} \log\oleft(\frac{R-\abs{x}}{1+\abs{x}}\tau\right).
	\]
\end{lemma}

\begin{proof}
	By definition
	\[
		\begin{split}
			W_\tau f_R(x,0)
			 & = \int_{\setR^d} f_R\bigl(x+\tau y\bigr) \overline{f_R\bigl(x-(1-\tau)y\bigr)} \d y \\
			 & = \int_{\setR^d} \abs{x+\tau y}^{-d/2} \bm1_{[1,R]}(\abs{x+\tau y})
			\abs{x-(1-\tau)y}^{-d/2} \bm1_{[1,R]}\bigl(\abs{x-(1-\tau)y}\bigr) \d y.
		\end{split}
	\]
	Consider the nonempty set
	\[
		A_{R,\abs{x},\tau} = \set*{y\in\setR^d}{\frac{1+\abs{x}}\tau \leq \abs{y} \leq R-\abs{x}}.
	\]
	Using $\tau\leq1-\tau\leq1$, for all $y\in A_{R,\abs{x},\tau}$ we have
	\begin{gather*}
		1 \leq \tau\abs{y}-\abs{x} \leq \abs{x+\tau y} \leq \abs{x}+\tau\abs{y} \leq \abs{x}+\abs{y} \leq R, \\
		1 \leq \tau\abs{y}-\abs{x} \leq (1-\tau)\abs{y}-\abs{x} \leq \abs{x-(1-\tau)y} \leq \abs{x}+\abs{y} \leq R,
	\end{gather*}
	therefore
	\[
		\bm1_{[1,R]}(\abs{x+\tau y}) \bm1_{[1,R]}\bigl(\abs{x-(1-\tau)y}\bigr) = 1;
	\]
	moreover, using $\abs{x}\leq\tau\abs{y}$, we have also
	\begin{align*}
		 & \abs{x+\tau y} \leq \abs{x}+\tau\abs{y} \leq 2\tau\abs{y},   &
		 & \abs{x-(1-\tau)y} \leq \abs{x}+(1-\tau)\abs{y} \leq \abs{y}.
	\end{align*}
	From these considerations, integrating in spherical coordinates we deduce that
	\[
		\begin{split}
			W_\tau f_R(x,0)
			 & \geq \int_{A_{R,\abs{x},\tau}} (2\tau\abs{y})^{-d/2} \abs{y}^{-d/2} \d y
			= (2\tau)^{-d/2} \int_{(1+\abs{x})/\tau}^{R-\abs{x}} r^{-d} \omega_{d-1} r^{d-1} \d r              \\
			 & = 2^{-d/2} \omega_{d-1} \tau^{-d/2} \log\oleft(\frac{R-\abs{x}}{1+\abs{x}}\tau\right). \qedhere
		\end{split}
	\]
\end{proof}

\begin{lemma}\label{lem:primal BJ lower bound}
	Let $R\geq1$ and $(x,\xi)\in\setR^d\times\setR^d$ such that
	\[
		8(R+\abs{x})\abs{\xi} \leq 1.
	\]
	\begin{enumerate}
		\item If $d=2$ and $R \geq 2+3\abs{x}$ then
		      \[
			      W_\BJ f_R(x,\xi)
			      \geq \pi \cos\bigl(4\pi(R+\abs{x})\abs{\xi}\bigr) \log\oleft(\frac12\frac{R-\abs{x}}{1+\abs{x}}\right)^2.
		      \]
		\item If $d>2$ and $R \geq 4+5\abs{x}$ then
		      \[
			      W_\BJ f_R(x,\xi)
			      \geq \frac{2^{3-d}\omega_{d-1}\log(2)}{d-2} \cos\bigl(4\pi(R+\abs{x})\abs{\xi}\bigr)
			      \left[\left(\frac{R-\abs{x}}{1+\abs{x}}\right)^{d/2-1}-2^{d-2}\right].
		      \]
	\end{enumerate}
\end{lemma}

\begin{proof}
	The assumptions of \autoref{lem:Re W_tau(x xi) lower bound} are satisfied, therefore
	\[
		\begin{split}
			W_\BJ f_R(x,\xi) = \Re W_\BJ f_R(x,\xi)
			 & = \int_0^1 \Re W_\tau f_R(x,\xi) \d\tau
			\geq \cos\bigl(4\pi(R+\abs{x})\abs{\xi}\bigr) \int_0^1 W_\tau f_R(x,0) \d\tau       \\
			 & = 2\cos\bigl(4\pi(R+\abs{x})\abs{\xi}\bigr) \int_0^{1/2} W_\tau f_R(x,0) \d\tau,
		\end{split}
	\]
	where in the last equality we used also that $W_\tau f_R(x,0)=W_{1-\tau} f_R(x,0)$ for every $\tau\in(0,1)$.
	Notice that $\cos\bigl(4\pi(R+\abs{x})\abs{\xi}\bigr) \geq 0$ since $4\pi(R+\abs{x})\abs{\xi} \leq \pi/2$,
	so we need to find a lower bound for the last integral.

	Both assumptions on $R$ ensure that $C\coloneqq\frac{R-\abs{x}}{1+\abs{x}}\geq2$.
	It is clear that for every $\tau\in[0,1/2]$ we have $W_\tau f_R(x,0)\geq0$; furthermore, for every $\tau\in[1/C,1/2]$
	we can apply \autoref{lem:W_tau(x 0) lower bound} and obtain
	\begin{equation}\label{eq:int W_tau lower bound}
		\int_0^{1/2} W_\tau f_R(x,0) \d\tau
		\geq \int_{1/C}^{1/2} W_\tau f_R(x,0) \d\tau
		\geq 2^{-d/2}\omega_{d-1} \int_{1/C}^{1/2} \tau^{-d/2} \log(C\tau) \d\tau.
	\end{equation}
	To conclude the treatment of this lower bound, we distinguish the two cases of the statement.

	\begin{enumerate}
		\item In this case the lower bound \eqref{eq:int W_tau lower bound} can be evaluated explicitly as
		      \[
			      \pi \int_{1/C}^{1/2} \tau^{-1} \log(C\tau) \d\tau
			      = \pi \int_1^{C/2} \frac{\log t}t \d t
			      = \pi \left[\frac{(\log t)^2}2\right]_1^{C/2}
			      = \frac\pi2 \log\oleft(\frac12\frac{R-\abs{x}}{1+\abs{x}}\right)^2,
		      \]
		      from which the thesis follows.

		\item The stricter assumption on $R$ ensures that $C\geq4$, therefore the lower bound
		      \eqref{eq:int W_tau lower bound} can be estimated as
		      \[
			      \begin{split}
				      2^{-d/2}\omega_{d-1} \int_{1/C}^{1/2} \tau^{-d/2}\log(C\tau) \d\tau
				       & = 2^{-d/2}\omega_{d-1} C^{d/2-1} \int_1^{C/2} t^{-d/2}\log(t) \d t                   \\
				       & \geq 2^{-d/2}\omega_{d-1} C^{d/2-1} \int_2^{C/2} t^{-d/2}\log(2) \d t                \\
				       & = 2^{-d/2}\omega_{d-1}\log(2) C^{d/2-1} \left[\frac{t^{1-d/2}}{1-d/2}\right]_2^{C/2} \\
				       & = \frac{2^{1-d/2}\omega_{d-1}\log(2)}{d-2} C^{d/2-1}
				      \left[2^{1-d/2}-(C/2)^{1-d/2}\right]                                                    \\
				       & = \frac{2^{2-d}\omega_{d-1}\log(2)}{d-2} \bigl(C^{d/2-1}-2^{d-2}\bigr),
			      \end{split}
		      \]
		      from which the thesis follows. \qedhere
	\end{enumerate}
\end{proof}

From the above results we deduce the following asymptotic lower bounds for $W_\BJ f_R$,
which will be key for proving \autoref{thm:critical d=2} and \autoref{thm:supercritical d>2}.

\begin{corollary}\label{cor:W_BJ lower bound}
	In any dimension $d\geq2$ there exists a constant $C_d>0$ such that for all $R$ sufficiently large one has:
	\begin{enumerate}
		\item if $d=2$
		      \begin{align*}
			      W_\BJ f_R(x,\xi) & \geq C_2 (\log R)^2                                      &
			                       & \forall (x,\xi) \in B(0,1) \times B\bigl(0,1/(9R)\bigr);
		      \end{align*}
		\item if $d>2$
		      \begin{align*}
			      W_\BJ f_R(x,\xi) & \geq C_d R^{d/2-1}
			                       & \forall (x,\xi) \in B(0,1) \times B\bigl(0,1/(9R)\bigr).
		      \end{align*}
	\end{enumerate}
\end{corollary}

\subsection{Proofs of the main theorems}

We are now almost in a position to proceed with the proofs of the main results. Since the blow-up
behavior in \autoref{thm:critical d=2} and \autoref{thm:supercritical d>2} arises from a
concentration phenomenon (namely, $\abs{W_\BJ f_R}$ diverges on a shrinking set), we first require a
variant of the Lebesgue density theorem for sets of product form, to ensure that a non-vanishing
fraction of this diverging contribution remains inside $\Omega$.

\begin{proposition}\label{prop:Lebesgue differentiation}
	Let $d\in\setN_+$ and $\Omega\subseteq\setR^d\times\setR^d$ be a measurable set with
	$0<\leb^{2d}(\Omega)$. Then for every $\rho>0$ there exist $(x_0,\xi_0)\in\setR^d\times\setR^d$ such that
	\[
		\liminf_{\sigma\to0} \frac{\leb^{2d}\bigl(B(x_0,\rho)\times B(\xi_0,\sigma)\cap\Omega\bigr)}
		{\leb^{2d}\bigl(B(x_0,\rho)\times B(\xi_0,\sigma)\bigr)} > 0.
	\]
\end{proposition}

\begin{proof}
	Since
	\[
		\leb^{2d}\oleft(\Omega \setminus \bigcup_{x\in\setZ^d}
		B\oleft(\frac\rho{\sqrt{d}}x,\rho\right)\times\setR^d\cap\Omega\right)
		= 0
	\]
	we deduce that there exists a point $x_0\in\setR^d$ such that
	$\Omega'=B(x_0,\rho)\times\setR^d\cap\Omega$ is non-negligible, i.e.\ $\leb^{2d}(\Omega')>0$.
	Define the function $m\in L^1_\loc(\setR^d)$ given by
	\begin{align*}
		m(\xi)      & = \frac{\leb^d(\Omega'_\xi)}{\leb^d\bigl(B(x_0,\rho)\bigr)}, &
		\Omega'_\xi & = \set{x\in B(x_0,\rho)}{(x,\xi)\in\Omega'}.
	\end{align*}
	We have that $m\geq0$ and $m$ is not identically zero, therefore by the Lebesgue differentiation theorem
	there exists $\xi_0\in\setR^d$ such that
	\[
		\liminf_{\sigma\to0} \dashint_{B(\xi_0,\sigma)} m(\xi) \d\xi > 0.
	\]
	The thesis follows from the observation that, by Fubini--Tonelli,
	\[
		\begin{split}
			\dashint_{B(\xi_0,\sigma)} m(\xi) \d\xi
			 & = \frac1{\leb^d\bigl(B(x_0,\rho)\bigr)\leb^d\bigl(B(\xi_0,\sigma)\bigr)}
			\int_{B(\xi_0,\sigma)} \leb^d(\Omega'_\xi) \d\xi                              \\
			 & = \frac{\leb^{2d}\bigl(B(x_0,\rho)\times B(\xi_0,\sigma)\cap\Omega'\bigr)}
			{\leb^{2d}\bigl(B(x_0,\rho)\times B(\xi_0,\sigma)\bigr)}                      \\
			 & = \frac{\leb^{2d}\bigl(B(x_0,\rho)\times B(\xi_0,\sigma)\cap\Omega\bigr)}
			{\leb^{2d}\bigl(B(x_0,\rho)\times B(\xi_0,\sigma)\bigr)}. \qedhere
		\end{split}
	\]
\end{proof}

Exchanging the roles of $x$ and $\xi$ in the previous proof, one obtains the symmetric statement: for every $\sigma>0$ there exists $(x_0,\xi_0)\in\setR^d\times\setR^d$ such that
\begin{equation}\label{eq:Lebesgue differentiation x}
	\liminf_{r\to0} \frac{\leb^{2d}\bigl(B(x_0,r)\times B(\xi_0,\sigma)\cap\Omega\bigr)}
	{\leb^{2d}\bigl(B(x_0,r)\times B(\xi_0,\sigma)\bigr)} > 0.
\end{equation}
This version will be used in \autoref{sec:attainment-supercritical}.

\subsubsection{Subcritical case in dimension \texorpdfstring{$d\geq2$}{d ≥ 2}}\label{sec:subcritical}

The treatment of the subcritical case is based on the techniques already employed in
\cite{Stra-Svela-Trapasso-2025} to address the case $p\in[1,\infty)$ in dimension $1$.
It uses the time-frequency profile decomposition introduced in \cite{Nicola-Romero-Trapasso-2022},
to which we refer the reader for details on the technique.

\begin{proof}[Proof of \autoref{thm:subcritical}]
	Let us first show that the supremum \(\sup_{f\in L^2(\setR^d)\setminus\{0\}} \frac{\norm{W_\BJ f}_{L^p(\Omega)}}{\norm{f}_{L^2(\setR^d)}^2}\) is finite. By the triangle inequality it is sufficient to consider \(\frac{\int_0^1\norm{W_{\tau} f}_{L^p(\Omega)}\d\tau}{\norm{f}_{L^2(\setR^d)}^2}.\) For the case \(p\in [1,2]\) H\"older's inequality and Moyal's identity for \(W_{\tau}\) give, for all \(\tau\in (0,1)\),
	\[
		\norm{W_\tau f}_{L^p(\Omega)}
		\leq \leb^{2d}(\Omega)^{\frac{1}{p}-\frac{1}{2}}\norm{W_\tau f}_{L^2(\setR^{2d})}
		= \leb^{2d}(\Omega)^{\frac{1}{p}-\frac{1}{2}}\norm{f}_{L^2(\setR^d)}^2
	\]
	which shows finiteness for \(p\in [1,2]\). For any \(p>2\) interpolation between \(2\) and \(\infty\) gives the estimate
	\[
		\norm{W_\tau f}_{L^p(\Omega)}
		\lesssim \norm{W_\tau f}_{L^2(\Omega)}^{2/p}\norm{W_\tau f}_{L^\infty(\Omega)}^{1-2/p}
		\leq (\tau(1-\tau))^{-\frac{d}{2}\left(1-\frac{2}{p}\right)}\norm{f}_{L^2(\setR^d)}^2,
	\]
	where we have used Moyal's identity and the Cauchy--Schwarz estimate \(\norm{W_\tau f}_{L^\infty(\setR^{2d})}\leq (\tau(1-\tau))^{-\frac{d}{2}}\norm{f}_{L^2(\setR^d)}^2.\) This function is integrable in \(\tau\) precisely when \(p<p_*(d),\) which shows finiteness. For future reference note that for any \(f,g\in L^2(\setR^d), \tau\in (0,1)\) and \(p\in [1,\infty)\) we have the estimate
	\[
		\norm{W_\tau (f,g)}_{L^p(\Omega)}\lesssim H_{p}(\tau;f,g,\Omega)
		= \begin{cases}\displaystyle
			\leb^{2d}(\Omega)^{\frac1p-\frac12}\|f\|_{L^2(\setR^d)}\|g\|_{L^2(\setR^d)}                & (1\leq p\leq2) \\
			(\tau(1-\tau))^{-\frac d2 \left(1-\frac2p\right)} \|f\|_{L^2(\setR^d)}\|g\|_{L^2(\setR^d)} & (p>2)
		\end{cases}
	\]
	and \(H_p\) is in \(L^1(0,1)\) if and only if \(p<p_*(d)\). We write $h_p(\tau)$ for the value of $H_p(\tau;f,g,\Omega)$ when $\norm{f}_{L^2(\setR^d)}=\norm{g}_{L^2(\setR^d)}=1$, hence $H_p(\tau;f,g,\Omega)=h_p(\tau)\norm{f}_{L^2(\setR^d)}\norm{g}_{L^2(\setR^d)}$ and $h_p\in L^1(0,1)$ for $p<p_*(d)$. Note also that the same bound holds with $\Omega$ replaced by any translate $\Omega-w$, $w\in\setR^{2d}$.

	Now let \(f^{(n)}\) be a sequence in \(L^2(\setR^d)\) converging weakly to \(f\). It is enough to consider the case \(f=0\), as the general case will then follow by bilinearity and covariance, see~\cite[Prop. 6.4]{Stra-Svela-Trapasso-2025}. We may also assume \(\norm{f^{(n)}}_{L^2(\setR^d)}\leq 1\) without loss of generality. We now perform a time-frequency profile decomposition on \(f^{(n)}\): for each fixed $k\in\setN$ we can write
	\[
		f^{(n)}=\sum_{j=1}^k \pi(z_j^{(n)}) \phi_j + w_k^{(n)},
	\]
	with asymptotic separation $|\zjn-\zjpn |\to\infty$ ($j\ne j'$), $\sum_{j=1}^\infty\|\phi_j\|_{L^2}^2\le 1$ and
	\[
		\lim_{k\to\infty} \limsup_{n\to\infty}\|w_k^{(n)}\|_{M^\infty}=0.
	\]
	Since $f^{(n)}\rightharpoonup 0$ all sequences \(\{\zjn\}\) diverge: if some $z_{j_0}^{(n)}$ were bounded in $\setR^{2d}$, then we would have $\pi(z_{j_0}^{(n)})^*f^{(n)}\rightharpoonup \phi_{j_0}\ne 0$, contradicting $f^{(n)}\rightharpoonup 0$. As a result,
	\[
		|z_j^{(n)}| \to \infty \quad \text{for every } j\ge 1.
	\]
	Moreover, by the asymptotic orthogonality of the profile decomposition we have
	\[
		\lim_{n\to\infty}\norm{\Fkn}_{L^2(\setR^d)}^2=\sum_{j=1}^k\norm{\phi_j}_{L^2(\setR^d)}^2\leq 1,
		\qquad \text{where } \Fkn\coloneqq\sum_{j=1}^k\pi(\zjn)\phi_j .
	\]
	Hence there exists $n_0(k)\in\setN$ such that $\norm{\Fkn}_{L^2(\setR^d)}\leq2$ and $\norm{\wkn}_{L^2(\setR^d)}\leq\norm{f^{(n)}}_{L^2(\setR^d)}+\norm{\Fkn}_{L^2(\setR^d)}\leq3$ for every $n\geq n_0(k)$. Consequently, for $n\geq n_0(k)$ all the bounds $H_p(\tau;\cdot,\cdot,\Omega)$ used below for the remainder terms are dominated by $9h_p(\tau)$, a majorant independent of $n$ and $k$. Now, by bilinearity we have
	\[
		\begin{split}
			W_\BJ(f^{(n)})
			 & = \sum_{j=1}^k W_\BJ(\pi(\zjn)\phi_j)
			+ \sum_{1\leq j\neq j' \leq k} W_\BJ(\pi(\zjn)\phi_j,\pi(\zjpn)\phi_{j'}) \\
			 & +W_\BJ(\Fkn,\wkn)+W_\BJ(\wkn,\Fkn)+W_\BJ(\wkn).
		\end{split}
	\]
	Let us proceed to control each term in the expansion.

	\smallskip
	\noindent \textbf{Step 1.} \textit{Vanishing of the pure profiles.}
	By covariance we want to show that
	\[
		\lim_{n\to \infty}\|W_\BJ(\pi(\zjn)\phi_j)\|_{L^p(\Omega)}
		= \lim_{n\to \infty}\|T_{\zjn}W_{\BJ}\phi_j\|_{L^p(\Omega)}
		= 0.
	\]
	By \(C_0\)-decay and covariance of \(W_{\tau}\) for all \(\tau\in (0,1)\), as well as inner regularity of Lebesgue measure, we therefore get
	\[
		\lim_{n\to \infty}\|T_{\zjn}W_{\tau}\phi_j\|_{L^p(\Omega)} = 0
	\]
	for any \(p\in [1,\infty)\). When \(p<p_*\) \(H_p(\tau)\) dominates \(\|T_{\zjn}W_{\tau}\phi_j\|_{L^p(\Omega)}\), so the triangle inequality and dominated convergence yield
	\[
		\begin{split}
			\lim_{n\to \infty}\|T_{\zjn}W_{\BJ}(\phi_j)\|_{L^p(\Omega)}
			 & \leq \lim_{n\to \infty}\int_0^1\|T_{\zjn}W_{\tau}(\phi_j)\|_{L^p(\Omega)}\d\tau \\
			 & = \int_0^1\lim_{n\to \infty}\|T_{\zjn}W_{\tau}(\phi_j)\|_{L^p(\Omega)}\d\tau
			= 0
		\end{split}
	\]
	for any profile \(\phi_j\).

	\smallskip
	\noindent \textbf{Step 2.} \textit{Vanishing of the cross terms.}
	This is completely analogous to the proof of~\cite[Lemma 6.7]{Stra-Svela-Trapasso-2025}, except that dominated convergence can only be invoked for \(p<p_*\).

	\smallskip
	\noindent \textbf{Step 3.} \textit{Vanishing of remainder terms.}
	The argument will be the same for \(W_{\BJ}(F_k^{(n)},w_k^{(n)})\), \(W_{\BJ}(w_k^{(n)},F_k^{(n)})\) and \(W_{\BJ}(w_k^{(n)},w_k^{(n)})\), so we will only show it for \(W_{\BJ}(F_k^{(n)},w_k^{(n)})\).
	By~\cite[Theorem 4.12]{CT20} and a slight modification of~\cite[Section 4]{Stra-Svela-Trapasso-2025} we get for \(\tau\in (0,1)\) that
	\[
		\lim_{k\to\infty}\limsup_{n\to\infty} \|W_{\tau}(F_k^{(n)},w_k^{(n)})\|_{L^p(\Omega)} = 0,
	\]
	again for all \(p\in [1,\infty)\). By the uniform bound on $\norm{F_k^{(n)}}_{L^2(\setR^d)}$ and $\norm{w_k^{(n)}}_{L^2(\setR^d)}$ established above, the function $9h_p(\tau)$ dominates \(\|W_{\tau}(F_k^{(n)},w_k^{(n)})\|_{L^p(\Omega)}\) for all \(\tau\in (0,1)\) and all $n\geq n_0(k)$. For \(p<p_*\) we now use, in order, the triangle inequality, the reverse Fatou lemma (which we can use because \(9h_p\) dominates) and dominated convergence to get
	\[
		\begin{split}
			\lim_{k\to\infty}\limsup_{n\to\infty}\|W_{\BJ}(F_k^{(n)},w_k^{(n)})\|_{L^p(\Omega)}
			 & \leq \lim_{k\to\infty}\limsup_{n\to\infty}\int_0^1\|W_{\tau}(F_k^{(n)},w_k^{(n)})\|_{L^p(\Omega)}\d\tau \\
			 & \leq \lim_{k\to\infty}\int_0^1\limsup_{n\to\infty}\|W_{\tau}(F_k^{(n)},w_k^{(n)})\|_{L^p(\Omega)}\d\tau \\
			 & = \int_0^1\lim_{k\to\infty}\limsup_{n\to\infty}\|W_{\tau}(F_k^{(n)},w_k^{(n)})\|_{L^p(\Omega)}\d\tau
			= 0.
		\end{split}
	\]

	\smallskip
	\noindent \textbf{Step 4.} \textit{Conclusion.}
	Since all the terms in the profile decomposition vanish in the limit, so must \(\norm{W_\BJ f^{(n)}}_{L^p(\Omega)}\) by the triangle inequality. We conclude that \(J_\BJ^p\) is sequentially weakly continuous on \(L^2(\setR^d)\).
\end{proof}

\begin{remark}
	It is easy to realize that the proof of \autoref{thm:subcritical} yields more than the mere
	existence of optimizers. Indeed, after setting
	\[
		L_{p,\Omega}\coloneqq \sup_{\norm{f}_{L^2(\setR^d)}=1}\norm{W_\BJ f}_{L^p(\Omega)}, \qquad p\in[1,p_*(d)),
	\]
	then every (normalized) maximizing sequence for $L_{p,\Omega}$ is relatively compact in $L^2(\setR^d)$.
	Indeed, if $(f_n)_{n \in \setN}$ is any such maximizing sequence, by weak compactness in
	$L^2(\setR^d)$ we may assume $f_n\rightharpoonup f$. The sequential weak continuity proved in
	\autoref{thm:subcritical} implies
	\[
		\norm{W_\BJ f}_{L^p(\Omega)} = \lim_{n\to\infty}\norm{W_\BJ f_n}_{L^p(\Omega)} = L_{p,\Omega}.
	\]
	If $\norm{f}_{L^2(\setR^d)}<1$, since\footnote{
		This follows by a straightforward argument. Given \(\varphi(x)=e^{-\pi |x|^2}\), we have that
		$W_\BJ\varphi(0,0)>0$ and \(W_\BJ\varphi\) is continuous near the origin, by dominated convergence
		applied to $W_\BJ\varphi(x,\xi)$ as $(x,\xi)\to 0$. Therefore \(W_\BJ\varphi\) is nonzero on some
		ball \(B(0,r)\), and the claim follows by covariance along a Lebesgue density point of $\Omega$.}
	$L_{p,\Omega}>0$, then by homogeneity
	\[
		\frac{\norm{W_\BJ f}_{L^p(\Omega)}}{\norm{f}_{L^2(\setR^d)}^2}
		> \norm{W_\BJ f}_{L^p(\Omega)} = L_{p,\Omega},
	\]
	contradicting the definition of $L_{p,\Omega}$. Hence $\norm{f}_{L^2(\setR^d)}=1$, so $f$ is a maximizer,
	and $f_n\to f$ strongly in $L^2(\setR^d)$.

	In turn, this produces a qualitative form of stability. More precisely, if $\mathcal M_{p,\Omega}$
	denotes the set of optimizers, then for every $\varepsilon>0$ there exists $\eta>0$ such that
	\[
		\text{if} \quad \norm{f}_{L^2(\setR^d)}=1
		\quad\text{and}\quad \norm{W_\BJ f}_{L^p(\Omega)}\ge L_{p,\Omega}-\eta \quad \text{then} \quad
		\inf_{g \in \mathcal M_{p,\Omega}} \norm{f-g}_{L^2(\setR^d)}<\varepsilon.
	\]
	Otherwise, one could construct a maximizing sequence staying at positive distance from
	$\mathcal M_{p,\Omega}$, contradicting the relative compactness just proved.
\end{remark}

\begin{remark}\label{rmk:subcritical-infinite-measure}
	The assumption $\leb^{2d}(\Omega)<\infty$ in \autoref{thm:subcritical} is used for the finiteness of the supremum only when $1\leq p<2$. Indeed, for $2<p<p_*(d)$ the estimate leading to $H_p$ in the proof of \autoref{thm:subcritical} is global: since $\norm{F}_{L^p}^p\leq\norm{F}_{L^\infty}^{p-2}\norm{F}_{L^2}^2$, Moyal's identity and the pointwise bound for $W_\tau$ give
	\[
		\norm{W_\BJ(f,g)}_{L^p(\setR^{2d})}
		\leq \int_0^1\bigl(\tau(1-\tau)\bigr)^{-\frac d2\left(1-\frac2p\right)}\d\tau\;\norm{f}_{L^2(\setR^d)}\norm{g}_{L^2(\setR^d)},
	\]
	and the integral is finite precisely because $p<p_*(d)$; the case $p=2$ is encompassed by \autoref{rmk:L2 estimate}. The supremum in \autoref{thm:subcritical} is therefore finite for every measurable $\Omega\subseteq\setR^{2d}$ and $p\in[2,p_*(d))$, including $\Omega=\setR^{2d}$ (see also~\cite{Wang26}). At the critical endpoint $p=p_*(d)$ with $d>2$ the same conclusion is given by \cite[Theorem 1.1]{Wang26}. On the other hand, sequential weak continuity may fail when $\leb^{2d}(\Omega)=\infty$: for $\Omega=\setR^{2d}$, $\varphi(x)=e^{-\pi\abs{x}^2}$ and $\abs{z_n}\to\infty$, the functions $f_n=\pi(z_n)\varphi$ satisfy $f_n\weakto0$, while $\norm{W_\BJ f_n}_{L^p(\setR^{2d})}=\norm{W_\BJ\varphi}_{L^p(\setR^{2d})}>0$ by covariance.
\end{remark}

\subsubsection{Critical case in dimension \texorpdfstring{$d=2$}{d = 2}}\label{sec:critical}

\begin{proof}[Proof of \autoref{thm:critical d=2}]
	Let $(x_0,\xi_0)\in\setR^2\times\setR^2$ be given by \autoref{prop:Lebesgue differentiation},
	such that
	for every $R$ sufficiently large the intersection $B(x_0,1)\times B\bigl(\xi_0,1/(9R)\bigr)\cap\Omega$
	is non-negligible. This ensures that the essential supremum of a function on this set is at most
	the essential supremum on $\Omega$. Define the function $g_R = \pi(x_0,\xi_0)f_R$,
	so that $\norm{g_R}_{L^2(\setR^2)}=\norm{f_R}_{L^2(\setR^2)}$ and
	\[
		\abs{W_\BJ g_R(x,\xi)}
		= \abs{W_\BJ f_R(x-x_0,\xi-\xi_0)}
		\geq W_\BJ f_R(x-x_0,\xi-\xi_0).
	\]
	By \autoref{cor:W_BJ lower bound} we have the lower bound
	\begin{align*}
		\abs{W_\BJ g_R(x,\xi)} & \geq C_2 (\log R)^2                                            &
		                       & \forall (x,\xi) \in B(x_0,1) \times B\bigl(\xi_0,1/(9R)\bigr),
	\end{align*}
	therefore, recalling \eqref{eq:f_R L^2 norm}, we obtain
	\[
		\sup_{f\in L^2(\setR^2)\setminus\{0\}} \frac{\norm{W_\BJ f}_{L^\infty(\Omega)}}{\norm{f}_{L^2(\setR^2)}^2}
		\geq \lim_{R\to\infty} \frac{\norm{W_\BJ g_R}_{L^\infty(\Omega)}}{\norm{g_R}_{L^2(\setR^2)}^2}
		\geq \lim_{R\to\infty} \frac{C_2 (\log R)^2}{\omega_1 \log R}
		= \infty. \qedhere
	\]
\end{proof}

\subsubsection{Supercritical case in dimension \texorpdfstring{$d>2$}{d > 2}}\label{sec:supercritical}

\begin{proof}[Proof of \autoref{thm:supercritical d>2}]
	Let $\eps>0$ and $(x_0,\xi_0)\in\setR^d\times\setR^d$ given by
	\autoref{prop:Lebesgue differentiation} such that
	\[
		\leb^{2d}\bigl(B(x_0,1)\times B\bigl(\xi_0,1/(9R)\bigr)\cap\Omega\bigr)
		\geq \eps
		\leb^{2d}\bigl(B(x_0,1)\times B\bigl(\xi_0,1/(9R)\bigr)\bigr)
	\]
	for every $R$ sufficiently large. Define the function $g_R = \pi(x_0,\xi_0)f_R$, so that
	$\norm{g_R}_{L^2(\setR^d)}=\norm{f_R}_{L^2(\setR^d)}$ and
	\[
		\abs{W_\BJ g_R(x,\xi)}
		= \abs{W_\BJ f_R(x-x_0,\xi-\xi_0)}
		\geq W_\BJ f_R(x-x_0,\xi-\xi_0).
	\]
	By \autoref{cor:W_BJ lower bound} we have the lower bound
	\begin{align*}
		\abs{W_\BJ g_R(x,\xi)} & \geq C_d R^{d/2-1}                                             &
		                       & \forall (x,\xi) \in B(x_0,1) \times B\bigl(\xi_0,1/(9R)\bigr),
	\end{align*}
	from which the case $p=\infty$ readily follows as in the proof of \autoref{thm:critical d=2}. For $p\in\bigl(p_*(d),\infty\bigr)$ we get
	\[
		\begin{split}
			\norm{W_\BJ g_R}_{L^p(\Omega)}
			 & \geq \norm{W_\BJ g_R}_{L^p(B(x_0,1)\times B(\xi_0,1/(9R))\cap\Omega)}                             \\
			 & \geq C_d R^{d/2-1} \leb^{2d}\bigl(B(x_0,1)\times B\bigl(\xi_0,1/(9R)\bigr)\cap\Omega\bigr)^{1/p}  \\
			 & \geq \eps^{1/p} C_d R^{d/2-1} \leb^{2d}\bigl(B(x_0,1)\times B\bigl(\xi_0,1/(9R)\bigr)\bigr)^{1/p} \\
			 & = 9^{-d/p} \eps^{1/p} C_d \left(\frac{\omega_{d-1}}{d}\right)^{2/p} R^{d/2-1-d/p}.
		\end{split}
	\]
	Recalling \eqref{eq:f_R L^2 norm}, we obtain
	\[
		\sup_{f\in L^2(\setR^d)\setminus\{0\}} \frac{\norm{W_\BJ f}_{L^p(\Omega)}}{\norm{f}_{L^2(\setR^d)}^2}
		\geq \lim_{R\to\infty} \frac{\norm{W_\BJ g_R}_{L^p(\Omega)}}{\norm{g_R}_{L^2(\setR^d)}^2}
		\gtrsim \lim_{R\to\infty} \frac{R^{d/2-1-d/p}}{\log R}
		= \infty,
	\]
	because
	\[
		d/2-1-d/p
		= d \left(\frac12-\frac1d-\frac1p\right)
		= d \left(\frac1{p_*(d)}-\frac1p\right)
		> 0
	\]
	for $p>p_*(d)$.
\end{proof}

\subsubsection{Critical case in dimension \texorpdfstring{$d>2$}{d > 2}}\label{sec:critical d>2}

The following proof is reminiscent of that of \cite[Proposition~3.6]{Stra-Svela-Trapasso-2025}, where the antipodally time-shifted pair of profiles is replaced by a reciprocally scaled one.

\begin{proof}[Proof of \autoref{thm:critical d>2}]
	For $r>0$ let $\phi_r(x)=r^{-d/2}e^{-\pi\abs{x}^2/r^2}$, so that $\norm{\phi_r}_{L^2(\setR^d)}^2=2^{-d/2}$.
	For $0<\eps\leq1/2$ consider the reciprocal-scale superposition
	\[
		f_\eps = \frac{\phi_\eps+\phi_{1/\eps}}{N_\eps},
		\qquad
		N_\eps \coloneqq \norm{\phi_\eps+\phi_{1/\eps}}_{L^2(\setR^d)}.
	\]
	Since $\angles{\phi_\eps,\phi_{1/\eps}}=(\eps^2+\eps^{-2})^{-d/2}\in\bigl(0,2^{-d/2}\bigr]$, we have
	$2^{1-d/2}\leq N_\eps^2\leq2^{2-d/2}$. Moreover $f_\eps\weakto0$ as $\eps\to0$: indeed
	$\norm{\phi_\eps}_{L^1}=\norm{\phi_{1/\eps}}_{L^\infty}=\eps^{d/2}$, so that
	\[
		\abs{\angles{f_\eps,h}}
		\leq N_\eps^{-1}\eps^{d/2}\bigl(\norm{h}_{L^\infty}+\norm{h}_{L^1}\bigr) \to 0
		\qquad \forall h\in L^1(\setR^d)\cap L^\infty(\setR^d),
	\]
	and the claim follows since $\norm{f_\eps}_{L^2}=1$ and $L^1\cap L^\infty$ is dense in $L^2(\setR^d)$.

	Since $W_\tau(g,f)=\overline{W_{1-\tau}(f,g)}$, the change of variables $\tau\mapsto1-\tau$ gives
	\begin{equation}\label{eq:three-terms}
		\Re W_\BJ f_\eps
		= \frac{1}{N_\eps^2} \int_0^1 \Re\bigl[
			W_\tau\phi_\eps+W_\tau\phi_{1/\eps}+2W_\tau(\phi_\eps,\phi_{1/\eps})
			\bigr]\d\tau.
	\end{equation}
	The integrands can be computed explicitly: for $r,s>0$ and $\tau\in[0,1]$, a Gaussian integration
	(cf.\ the proof of \cite[Lemma~2.4]{Wang26}) gives
	\[
		W_\tau(\phi_r,\phi_s)(x,\xi)
		= \left(\frac{rs}{D}\right)^{d/2} \exp\oleft(-\pi\frac{\abs{x}^2+r^2s^2\abs{\xi}^2}D\right) e^{2\pi ibx\cdot\xi},
	\]
	where
	\[
		D = r^2(1-\tau)^2+s^2\tau^2
		\quad\text{and}\quad
		b = \frac{s^2\tau-r^2(1-\tau)}D.
	\]
	By the Cauchy--Schwarz inequality,
	\[
		1 = (1/r,1/s) \cdot \bigl(r(1-\tau),s\tau\bigr)
		\leq \sqrt{\frac1{r^2}+\frac1{s^2}} \sqrt{D},
		\qquad\text{that is}\qquad
		D \geq \frac{r^2s^2}{r^2+s^2},
	\]
	and therefore
	\[
		\abs{b}
		= \frac{\abs{(-r,s)\cdot\bigl(r(1-\tau),s\tau\bigr)}}D
		\leq \frac{\sqrt{r^2+s^2}}{\sqrt{D}}
		\leq \frac{r^2+s^2}{rs}.
	\]
	In particular $\abs{W_\tau(\phi_r,\phi_s)}\leq\bigl((r^2+s^2)/(rs)\bigr)^{d/2}$ uniformly in $\tau$,
	so the integral in \eqref{eq:three-terms} converges absolutely at every point.

	The pairs $(r,s)=(\eps,\eps)$ and $(r,s)=(1/\eps,1/\eps)$ give $\abs{b}\leq2$, while
	$(r,s)=(\eps,1/\eps)$ gives $\abs{b}\leq\eps^2+\eps^{-2}$; since $\eps\leq1$, in all three cases
	$\abs{b}\leq2\eps^{-2}$. Consider the phase-space box
	\[
		Q_\eps \coloneqq B\bigl(0,\eps/\sqrt{12}\bigr)\times B\bigl(0,\eps/\sqrt{12}\bigr) \subseteq \setR^d\times\setR^d.
	\]
	For every $(x,\xi)\in Q_\eps$ and $\tau\in(0,1)$ the phases of the three terms in
	\eqref{eq:three-terms} satisfy $\abs{2\pi bx\cdot\xi}\leq 2\pi\cdot2\eps^{-2}\cdot\eps^2/12=\pi/3$,
	hence all three integrands are nonnegative on $Q_\eps$. We may therefore drop the diagonal terms
	and restrict the cross term to $\tau\in[\eps^2,2\eps^2]\subseteq(0,1)$. For $(r,s)=(\eps,1/\eps)$
	we have $rs=1$ and $D=\eps^2(1-\tau)^2+\eps^{-2}\tau^2$, so that
	\[
		\eps^2 \leq \eps^{-2}\tau^2 \leq D \leq \eps^2+4\eps^2 = 5\eps^2
		\qquad \forall \tau\in[\eps^2,2\eps^2].
	\]
	Since moreover $\abs{x}^2+\abs{\xi}^2<\eps^2/6$ on $Q_\eps$, for all $(x,\xi)\in Q_\eps$ and
	$\tau\in[\eps^2,2\eps^2]$ we get
	\[
		\Re W_\tau(\phi_\eps,\phi_{1/\eps})(x,\xi)
		\geq D^{-d/2} \exp\oleft(-\pi\frac{\abs{x}^2+\abs{\xi}^2}{D}\right) \cos(\pi/3)
		\geq \frac12 5^{-d/2}e^{-\pi/6}\eps^{-d}.
	\]
	Integrating over $\tau\in[\eps^2,2\eps^2]$ and recalling $N_\eps^2\leq2^{2-d/2}$ we conclude
	\begin{equation}\label{eq:critical-lower-bound}
		\abs{W_\BJ f_\eps(x,\xi)}
		\geq \Re W_\BJ f_\eps(x,\xi)
		\geq \frac{2}{N_\eps^2} \int_{\eps^2}^{2\eps^2} \Re W_\tau(\phi_\eps,\phi_{1/\eps})(x,\xi)\d\tau
		\geq c_d\eps^{2-d}
		\qquad \forall (x,\xi)\in Q_\eps,
	\end{equation}
	with $c_d\coloneqq2^{d/2-2}5^{-d/2}e^{-\pi/6}$. This lower bound has exactly the critical size:
	since $\leb^{2d}(Q_\eps)=\kappa_d\eps^{2d}$ with $\kappa_d\coloneqq12^{-d}(\omega_{d-1}/d)^2$ and
	$1/p_*(d)=(d-2)/(2d)$, the product $\eps^{2-d}\leb^{2d}(Q_\eps)^{1/p_*(d)}=\kappa_d^{1/p_*(d)}$
	does not depend on $\eps$.

	Finally, let $z_0$ be a Lebesgue density point of $\Omega$ and set $g_\eps=\pi(z_0)f_\eps$; since
	$\pi(z_0)$ is unitary, $g_\eps\weakto0$. As $Q_\eps\subseteq B(0,\eps/\sqrt6)\subseteq\setR^{2d}$ and
	$\leb^{2d}\bigl(B(0,\eps/\sqrt6)\bigr)\asymp\eps^{2d}\asymp\leb^{2d}(Q_\eps)$, the Lebesgue density
	theorem yields
	\[
		\leb^{2d}\bigl((z_0+Q_\eps)\setminus\Omega\bigr)
		\leq \leb^{2d}\bigl(B(z_0,\eps/\sqrt6)\setminus\Omega\bigr)
		= o(\eps^{2d}),
	\]
	hence $\leb^{2d}\bigl((z_0+Q_\eps)\cap\Omega\bigr)\geq\frac12\leb^{2d}(Q_\eps)$ for every $\eps$
	small enough. By covariance $W_\BJ g_\eps(z)=W_\BJ f_\eps(z-z_0)$, so \eqref{eq:critical-lower-bound} gives
	\[
		\norm{W_\BJ g_\eps}_{L^{p_*(d)}(\Omega)}
		\geq c_d\eps^{2-d}\leb^{2d}\bigl((z_0+Q_\eps)\cap\Omega\bigr)^{1/p_*(d)}
		\geq c_d \left(\frac{\kappa_d}2\right)^{1/p_*(d)}
	\]
	for all $\eps$ small enough. Along any sequence $\eps_n\to0$ we thus have $g_{\eps_n}\weakto0$ and
	\[
		\liminf_{n\to\infty} J_\BJ^{p_*(d)}(g_{\eps_n}) > 0 = J_\BJ^{p_*(d)}(0),
	\]
	which disproves the sequential weak upper semicontinuity.
\end{proof}

\begin{remark}\label{rmk:truncated not uniform}
	Since the incomplete functionals $f\mapsto\norm{W_\BJ^\delta f}_{L^{p_*(d)}(\Omega)}$ introduced in \autoref{sec:incomplete} are sequentially weakly continuous when $0<\leb^{2d}(\Omega)<\infty$, \autoref{thm:critical d>2} shows that they cannot converge to $J_\BJ^{p_*(d)}$ uniformly on the unit ball of $L^2(\setR^d)$ as $\delta\downarrow0$ --- otherwise $J_\BJ^{p_*(d)}$ would be sequentially weakly continuous as well since weakly convergent sequences are bounded.
\end{remark}

\begin{remark}\label{rmk:critical-constant}
	The proof of \autoref{thm:critical d>2} also gives a lower bound for the critical concentration value which does not depend on $\Omega$. Indeed, set $L_{p,\Omega}\coloneqq\sup_{\norm{f}_{L^2(\setR^d)}=1}\norm{W_\BJ f}_{L^p(\Omega)}$ for every $p\in[1,\infty]$, extending the notation of the remark after the proof of \autoref{thm:subcritical}. Since $\leb^{2d}\bigl((z_0+Q_\eps)\cap\Omega\bigr)/\leb^{2d}(Q_\eps) \to 1$, we have
	\[
		L_{p_*(d),\Omega}\geq c_d\,\kappa_d^{1/p_*(d)}
		\qquad\text{for every measurable }\Omega\subseteq\setR^{2d}\text{ with }\leb^{2d}(\Omega)>0.
	\]
	This is in sharp contrast with the subcritical regime: for $p<p_*(d)$, H\"older's inequality and~\cite[Theorem 1.1]{Wang26} give $L_{p,\Omega}\leq C_d\leb^{2d}(\Omega)^{1/p-1/p_*(d)}$, which vanishes as $\leb^{2d}(\Omega)\to0$. In other words, critical concentration cannot be made small by shrinking $\Omega$, in accordance with the scale invariance at $p=p_*(d)$.

	Moreover, for $0<\leb^{2d}(\Omega)<\infty$ the concentration value is continuous from below at the critical exponent, in every dimension $d\geq1$:
	\[
		\lim_{p\uparrow p_*(d)}L_{p,\Omega}=L_{p_*(d),\Omega}.
	\]
	Indeed, H\"older's inequality gives $L_{p,\Omega}\leq\leb^{2d}(\Omega)^{1/p-1/p_*(d)}L_{p_*(d),\Omega}$. Conversely, for every normalized $f\in L^2(\setR^d)$ we have $\liminf_{p\uparrow p_*(d)}\norm{W_\BJ f}_{L^p(\Omega)}\geq\norm{W_\BJ f}_{L^{p_*(d)}(\Omega)}$, by Fatou's lemma if $p_*(d)<\infty$ and by the elementary limit $\norm{F}_{L^p(\Omega)}\to\norm{F}_{L^\infty(\Omega)}$ on sets of finite measure if $p_*(d)=\infty$; as a result, $\liminf_{p\uparrow p_*(d)}L_{p,\Omega}\geq\norm{W_\BJ f}_{L^{p_*(d)}(\Omega)}$, and it remains to take the supremum over $f$. For $d=1$ the limit is the unattained value $\pi$ of \autoref{thm:BJ d=1}, for $d=2$ it is $+\infty$ by \autoref{thm:critical d=2} (consistently with the asymptotics $L_{p,\Omega}\asymp p$ recorded in \autoref{sec:incomplete}), while for $d>2$ it is the finite value established in~\cite{Wang26}. In all dimensions $d>2$, $L_{p,\Omega}=\infty$ for every $p>p_*(d)$ by \autoref{thm:supercritical d>2}.
\end{remark}

\subsection{Proofs of attainment}\label{sec:attainment}

\subsubsection{Indirect proof of attainment: Baire category theorem}\label{sec:indirect-baire}

In this section we prove \autoref{thm:attainment-baire} and \autoref{cor:generic} via the Baire category theorem.
For a measurable set $\Omega\subseteq\setR^{2d}$ and $p\in[1,\infty]$ we define the functional
\[
	J_\BJ^p(f) \coloneqq \norm{W_\BJ f}_{L^p(\Omega)} \in [0,\infty],
	\qquad f\in L^2(\setR^d).
\]
The central idea is showing that, for $N\in\setN$, the sublevel sets
\[
	F_N \coloneqq \set[\big]{f\in L^2(\setR^d)}{J_\BJ^p(f)\leq N}
\]
are closed and with empty interior: the former property is proved in the next lemma,
whereas the latter follows from \autoref{lem:J-interior} and \autoref{thm:critical d=2}, \autoref{thm:supercritical d>2}.
With regard to the next lemma, see also \autoref{cor:bj-critical-weak-pairings} for a strengthening in the critical case.

\begin{lemma}\label{lem:J-lsc}
	Let $\Omega\subseteq\setR^{2d}$ be measurable and $p\in[1,\infty]$. Then the map
	$f\mapsto W_\BJ f$ is norm continuous from $L^2(\setR^d)$ to $L^2(\setR^{2d})$, and
	$J_\BJ^p$ is norm lower semicontinuous on $L^2(\setR^d)$. In particular, every
	sublevel set $F_N$ is closed.
\end{lemma}

\begin{proof}
	By sesquilinearity, $W_\BJ f-W_\BJ g = W_\BJ(f-g,f)+W_\BJ(g,f-g)$, hence
	\autoref{rmk:L2 estimate} gives
	\[
		\norm{W_\BJ f-W_\BJ g}_{L^2(\setR^{2d})}
		\leq \norm{f-g}_{L^2(\setR^d)}
		\bigl(\norm{f}_{L^2(\setR^d)}+\norm{g}_{L^2(\setR^d)}\bigr).
	\]
	Let $f_n\to f$ in $L^2(\setR^d)$ and choose a subsequence $(f_{n_k})$ such that
	$J_\BJ^p(f_{n_k})\to\liminf_nJ_\BJ^p(f_n)$. Since
	$W_\BJ f_{n_k}\to W_\BJ f$ in $L^2(\setR^{2d})$, after passing to a further
	subsequence we may assume that $W_\BJ f_{n_k}\to W_\BJ f$ almost everywhere.
	If $p<\infty$, Fatou's lemma yields
	\[
		\int_\Omega \abs{W_\BJ f}^p
		\leq \liminf_{k\to\infty}\int_\Omega\abs{W_\BJ f_{n_k}}^p;
	\]
	if $p=\infty$, for almost every $z\in\Omega$ we have
	$\abs{W_\BJ f(z)}=\lim_k\abs{W_\BJ f_{n_k}(z)}\leq\liminf_kJ_\BJ^\infty(f_{n_k})$.
	In both cases $J_\BJ^p(f)\leq\liminf_nJ_\BJ^p(f_n)$.
\end{proof}

\begin{lemma}\label{lem:J-interior}
	Let $\Omega\subseteq\setR^{2d}$ be measurable and $p\in[1,\infty]$. If some $F_N$
	has nonempty interior in $L^2(\setR^d)$, then there exists $C>0$ such that
	\[
		\norm{W_\BJ h}_{L^p(\Omega)}\leq C\norm{h}_{L^2(\setR^d)}^2
		\qquad \forall h\in L^2(\setR^d).
	\]
\end{lemma}

\begin{proof}
	Suppose that $F_N$ contains the closed ball of radius $r>0$ centred at
	$f_0\in L^2(\setR^d)$. By sesquilinearity we have the parallelogram identity
	\[
		W_\BJ(f_0+h)+W_\BJ(f_0-h) = 2W_\BJ f_0+2W_\BJ h,
	\]
	so for $\norm{h}_{L^2(\setR^d)}\leq r$ Minkowski's inequality gives
	\[
		J_\BJ^p(h)
		\leq \tfrac12J_\BJ^p(f_0+h)+\tfrac12J_\BJ^p(f_0-h)
		+J_\BJ^p(f_0)
		\leq 2N,
	\]
	because $f_0$ and $f_0\pm h$ all belong to $F_N$. Since
	$J_\BJ^p(\lambda h)=\abs{\lambda}^2J_\BJ^p(h)$ for $\lambda\in\setC$,
	applying this bound to $rh/\norm{h}_{L^2(\setR^d)}$ gives the claim with
	$C=2N/r^2$.
\end{proof}

\begin{proof}[Proof of \autoref{thm:attainment-baire}]
	Let $\Omega$ and $p$ be as in \autoref{thm:critical d=2} (with $p=\infty$) or as in
	\autoref{thm:supercritical d>2}. By \autoref{lem:J-lsc} every $F_N$ is closed.
	By \autoref{lem:J-interior}, a sublevel set with nonempty interior would give a
	bound $\norm{W_\BJ h}_{L^p(\Omega)}\lesssim\norm{h}_{L^2(\setR^d)}^2$, contradicting
	\autoref{thm:critical d=2} or \autoref{thm:supercritical d>2}; hence every $F_N$
	has empty interior. Therefore
	\[
		\set[\big]{f\in L^2(\setR^d)}{W_\BJ f\notin L^p(\Omega)}
		= \bigcap_{N\in\setN} \bigl(L^2(\setR^d)\setminus F_N\bigr)
	\]
	is a countable intersection of open dense subsets of $L^2(\setR^d)$. By the Baire
	category theorem it is a dense $G_\delta$ set, and in particular it is nonempty.
\end{proof}

\begin{proof}[Proof of \autoref{cor:generic}]
	Let $(B_k)_{k\in\setN}$ be an enumeration of the open balls in $\setR^{2d}$ with
	rational centres and radii. If $d=2$ let $p_m=\infty$ for all $m$; if $d>2$ let
	$(p_m)_{m\in\setN}\subset P_d$ be a sequence with $p_m\downarrow p_*(d)$. By
	\autoref{thm:attainment-baire}, applied with $\Omega=B_k$ and $p=p_m$, each set
	$G_{k,m}\coloneqq\set{f\in L^2(\setR^d)}{W_\BJ f\notin L^{p_m}(B_k)}$ is a dense
	$G_\delta$, hence so is $G\coloneqq\bigcap_{k,m}G_{k,m}$.
	Let $f\in G$, let $p\in P_d$, let $U\in\mathcal{U}_{2d}$, and let $Z\subset\setR^{2d}$ be a negligible set such that $U\cup Z$ has nonempty interior.
	Choose $k$ with $B_k\subseteq U\cup Z$ and $m$ with $p_m\leq p$. Since
	$\leb^{2d}(B_k)<\infty$, Hölder's inequality gives
	$L^p(B_k)\subseteq L^{p_m}(B_k)$, hence $W_\BJ f\notin L^p(B_k)$ and a fortiori
	$W_\BJ f\notin L^p(U)$, since $B_k\setminus U\subseteq Z$ is negligible.
\end{proof}

\begin{remark}
	The sets of $\mathcal U_{2d}$ cannot be replaced by arbitrary measurable sets of positive measure in
	\autoref{cor:generic}. Indeed, for every $f\in L^2(\setR^d)$ the function $W_\BJ f$
	belongs to $L^2(\setR^{2d})$ and is therefore finite almost everywhere, so
	$\Omega_M\coloneqq\set{z\in\setR^{2d}}{\abs{W_\BJ f(z)}\leq M}$ has positive measure
	for $M$ large enough, while $\abs{W_\BJ f}^p\leq M^{p-2}\abs{W_\BJ f}^2$ on $\Omega_M$
	shows that $W_\BJ f\in L^p(\Omega_M)$ for every $p\in[2,\infty)$, and trivially $\norm{W_\BJ f}_{L^\infty(\Omega_M)}\leq M$.
\end{remark}

\begin{remark}\label{rmk:dense-meager}
	The genericity in \autoref{thm:attainment-baire} and \autoref{cor:generic} is sharp, in the sense that the set of \enquote{good} functions is still dense. Indeed, let $\varphi(x)=e^{-\pi\abs{x}^2}$ and let $V\subseteq L^2(\setR^d)$ be the linear span of $\set{\pi(z)\varphi}{z\in\setR^{2d}}$, which is dense in $L^2(\setR^d)$ because a function orthogonal to all time-frequency shifts of $\varphi$ has vanishing short-time Fourier transform. A straightforward estimation shows that
	\[
		\norm{W_\tau(\pi(z_1)\varphi,\pi(z_2)\varphi)}_\infty
		\leq \bigl(\tau^2+(1-\tau)^2\bigr)^{-d/2}
		\leq 2^{d/2}
		\qquad \forall z_1,z_2\in\setR^{2d},\ \forall \tau\in(0,1),
	\]
	hence $W_\BJ f\in L^\infty(\setR^{2d})$ for every $f\in V$, and therefore $W_\BJ f\in L^2(\setR^{2d})\cap L^\infty(\setR^{2d})\subseteq L^p(\setR^{2d})$ for every $p\in[2,\infty]$. As a consequence, under the assumptions of \autoref{thm:critical d=2} or \autoref{thm:supercritical d>2}, the set $\set{f\in L^2(\setR^d)}{W_\BJ f\in L^p(\Omega)}$ is dense in $L^2(\setR^d)$ but meager, its complement being a dense $G_\delta$ set.
\end{remark}

\subsubsection{Indirect proof of attainment: Banach--Steinhaus and closed graph theorems}\label{sec:indirect}

In this section we prove the following statement, which is a weaker version of \autoref{thm:attainment-baire},
with an argument based on the Banach--Steinhaus theorem and the closed graph theorem.

\begin{theorem}\label{thm:attainment}
	Under the assumptions of \autoref{thm:critical d=2} or \autoref{thm:supercritical d>2}, there
	exists $f\in L^2(\setR^d)$ such that $W_\BJ f\not\in L^p(\Omega)$; in other words, the suprema
	in \autoref{thm:critical d=2} and \autoref{thm:supercritical d>2} are attained.
\end{theorem}

\begin{proof}[Proof of \autoref{thm:attainment}]
	We can assume without loss of generality $\leb^{2d}(\Omega)<\infty$, otherwise it is enough to
	choose a measurable set \(E\subseteq\Omega\) with \(0<\leb^{2d}(E)<\infty\) and prove that
	\(W_\BJ f\notin L^p(E)\), implying \(W_\BJ f\notin L^p(\Omega)\).
	Assume by contradiction that $W_\BJ f\in L^p(\Omega)$ for all $f\in L^2(\setR^d)$.
	By the polarization identity
	\[
		4W_\BJ(f,g) = W_\BJ(f+g) - W_\BJ(f-g) + iW_\BJ(f+ig) - iW_\BJ(f-ig)
	\]
	and Minkowski's inequality we have $W_\BJ(f,g)\in L^p(\Omega)$ for all $f,g\in L^2(\setR^d)$.

	For $g\in L^2(\setR^d)$, let $T_g:L^2(\setR^d)\to L^p(\Omega)$ be the linear map $T_g(f)=W_\BJ(f,g)$.
	We claim that $T_g$ has closed graph. Indeed, $T_g:L^2(\setR^d)\to L^2(\Omega)$ is continuous
	because
	\[
		\norm{T_g(f)}_{L^2(\Omega)}
		= \norm{W_\BJ (f,g)}_{L^2(\Omega)}
		\leq \norm{f}_{L^2(\setR^d)} \norm{g}_{L^2(\setR^d)}
	\]
	(see \autoref{rmk:L2 estimate} or~\cite[Theorem 1.4]{CdGN18}), hence the graph
	\[
		\graph(T_g)
		= \set*{\bigl(f,T_g(f)\bigr)}{f\in L^2(\setR^d)}
		\subset L^2(\setR^d) \times L^p(\Omega)
	\]
	is closed with respect to the coarser topology induced by the immersion in
	$L^2(\setR^d)\times L^2(\Omega)$, and therefore with respect to the native product topology too.
	By the closed graph theorem this now means that $T_g:L^2(\setR^d)\to L^p(\Omega)$ is continuous,
	i.e.\ $\norm{T_g(f)}_{L^p(\Omega)} \leq C(g) \norm{f}_{L^2(\setR^d)}$.

	Since $W_\BJ(f,g)=\overline{W_\BJ(g,f)}$, we have also
	\[
		\norm{T_gf}_{L^p(\Omega)}
		= \norm{W_\BJ(f,g)}_{L^p(\Omega)}
		= \norm{W_\BJ(g,f)}_{L^p(\Omega)}
		= \norm{T_fg}_{L^p(\Omega)}
		\leq C(f) \norm{g}_{L^2(\setR^d)}.
	\]

	Consider the family of continuous linear operators
	\[
		\mathcal{T}
		= \set{T_g}{g\in L^2(\setR^d),\ \norm{g}_{L^2(\setR^d)}\leq1}
		\subset \mathcal{B}\bigl(L^2(\setR^d),L^p(\Omega)\bigr).
	\]
	It is pointwise bounded because for every $f\in L^2(\setR^d)$ we have
	\[
		\sup_{T\in\mathcal{T}} \norm{Tf}_{L^p(\Omega)}
		= \sup_{\substack{g\in L^2(\setR^d) \\ \norm{g}_{L^2(\setR^d)}\leq1}} \norm{T_gf}_{L^p(\Omega)}
		\leq \sup_{\substack{g\in L^2(\setR^d) \\ \norm{g}_{L^2(\setR^d)}\leq1}} C(f) \norm{g}_{L^2(\setR^d)}
		\leq C(f).
	\]
	By Banach--Steinhaus theorem the family of operators is uniformly bounded, that is
	\[
		C
		\coloneqq \sup_{\substack{g\in L^2(\setR^d) \\ \norm{g}_{L^2(\setR^d)}\leq1}}
		\norm{T_g}_{B\bigl(L^2(\setR^d),L^p(\Omega)\bigr)}
		< \infty.
	\]
	This means that
	\[
		\norm{W_\BJ(f,g)}_{L^p(\Omega)} \leq C \norm{f}_{L^2(\setR^d)} \norm{g}_{L^2(\setR^d)},
	\]
	which in particular implies
	\[
		\norm{W_\BJ f}_{L^p(\Omega)} \leq C \norm{f}_{L^2(\setR^d)}^2,
	\]
	contradicting \autoref{thm:critical d=2} or \autoref{thm:supercritical d>2}.
\end{proof}

\begin{remark}\label{rmk:L2 estimate}
	The $\tau$-Wigner distribution can be written as a Fourier transform:
	\begin{align*}
		W_\tau(f,g)(x,\xi) & = \mathcal{F}h_x(\xi),                             &
		h_x(y)             & = f(x+\tau y) \overline{g\bigl(x-(1-\tau)y\bigr)}.
	\end{align*}
	By Plancherel identity, its squared $L^2$ norm is
	\[
		\begin{split}
			\norm{W_\tau(f,g)}_{L^2(\setR^{2d})}^2
			 & = \int_{\setR^d} \norm{\mathcal{F}h_x}_{L^2(\setR^d)}^2 \d x
			= \int_{\setR^d} \norm{h_x}_{L^2(\setR^d)}^2 \d x                                                   \\
			 & = \int_{\setR^d} \int_{\setR^d} \abs*{f(x+\tau y)}^2 \abs*{g\bigl(x-(1-\tau)y\bigr)}^2 \d y \d x
		\end{split}
	\]
	which, by the linear change of variables $(u,v)=\bigl(x+\tau y,x-(1-\tau)y\bigr)$ with Jacobian
	\[
		\abs*{\det\begin{pmatrix}I & \tau I \\ I & -(1-\tau)I\end{pmatrix}}
		= \abs*{\det\begin{pmatrix}I & \tau I \\ I-I & -(1-\tau)I-\tau I\end{pmatrix}}
		= \abs*{\det\begin{pmatrix}I & \tau I \\ 0 & -I\end{pmatrix}}
		= 1,
	\]
	is equal to
	\[
		\begin{split}
			\int_{\setR^d} \int_{\setR^d} \abs{f(u)}^2 \abs{g(v)}^2 \d u \d v
			= \norm{f}_{L^2(\setR^d)}^2 \norm{g}_{L^2(\setR^d)}^2.
		\end{split}
	\]
	Integrating in $\tau$ we finally get
	\[
		\norm{W_\BJ(f,g)}_{L^2(\setR^{2d})}^2
		\leq \int_0^1 \norm{W_\tau(f,g)}_{L^2(\setR^{2d})}^2 \d\tau
		\leq \norm{f}_{L^2(\setR^d)}^2 \norm{g}_{L^2(\setR^d)}^2.
	\]
\end{remark}

\subsubsection{Explicit proof of attainment: critical case in dimension \texorpdfstring{$d=2$}{d = 2}}
\label{sec:attainment-critical}

In the current and next sections we construct explicit examples of functions proving \autoref{thm:attainment}.

For the case $d=2$, the construction will make use of the functions $f_{r,R}$ defined in \autoref{def:f_rR},
whose squared $L^2$ norm is
\begin{equation}\label{eq:f_rR L^2 norm}
	\norm{f_{r,R}}_{L^2(\setR^d)}^2
	= \int_{\setR^d} \abs{x}^{-d} \bm1_{[r,R]}(\abs{x}) \d x
	= \int_r^R \rho^{-d} \omega_{d-1}\rho^{d-1} \d\rho
	= \omega_{d-1} \log(R/r).
\end{equation}

Although we are currently able to reach the conclusion only in dimension $d=2$, we state the preliminary lemmas in any dimension.

\begin{lemma}\label{lem:truncate_y}
	Let $R_0,R_1,R>0$ be some positive radii and let $f,g\in L^2(\setR^d)$ be such that
	$\supp f\subseteq B_{R_0}$ and $\supp g\subseteq B_{R_1}$.
	Then for all $x\in B_R$ and $\xi\in\setR^d$ we have
	\[
		W_\tau(f,g)(x,\xi)
		= \int_{B_{R_0+R_1+2R}} e^{-2\pi i\xi\cdot y} f(x+\tau y) \overline{g\bigl(x-(1-\tau)y\bigr)}
		\d y.
	\]
\end{lemma}

\begin{proof}
	In the integral defining $W_\tau(f,g)(x,\xi)$, the integrand function is nonzero if and only if both
	$\abs{x+\tau y}\leq R_0$ and $\abs{x-(1-\tau)y}\leq R_1$. By the reverse triangle inequality, from these
	we deduce $\tau\abs{y}-\abs{x}\leq R_0$ and $(1-\tau)\abs{y}-\abs{x}\leq R_1$. Summing up and using
	$\abs{x}\leq R$ we get $\abs{y}\leq R_0+R_1+2\abs{x}\leq R_0+R_1+2R$. Therefore we can restrict
	the defining integral to the ball with this radius.
\end{proof}

By invoking~\autoref{lem:truncate_y} we deduce the following estimates as straightforward
generalizations of Lemmas~\ref{lem:Re W_tau(x xi) lower bound}~to~\ref{lem:primal BJ lower bound}
and~\autoref{cor:W_BJ lower bound}.

\begin{lemma}\label{lem: cross-Re W_tau(x xi) lower bound}
	Let $R_0,R_1,R>0$ be some positive radii and let $f,g\in L^2\bigl(\setR^d;[0,\infty)\bigr)$ be such that
	$\supp f\subseteq B_{R_0}$ and $\supp g\subseteq B_{R_1}$.
	Then for all $x\in B_R$ and $\xi\in\setR^d$ with
	\[
		\abs{\xi} \leq \frac{\pi/3}{2\pi(R_0+R_1+2R)}
	\]
	we have
	\[
		\Re W_\tau(f,g)(x,\xi) \geq \frac12 W_\tau(f,g)(x,0).
	\]
\end{lemma}

\begin{lemma}\label{lem:two-sided W_tau(x 0) lower bound}
	Let $0<r<R$. If $x\in\setR^d$ and $\tau\in(0,1/2]$ satisfy
	\[
		\frac{r+\abs{x}}\tau \leq R-\abs{x}
	\]
	then
	\[
		W_\tau f_{r,R}(x,0)
		\geq 2^{-d/2}\omega_{d-1}\tau^{-d/2} \log\oleft(\frac{R-\abs{x}}{r+\abs{x}}\tau\right).
	\]
\end{lemma}

The analogue of~\autoref{cor:W_BJ lower bound} we only state for \(d=2\), since this case is the case we will need.

\begin{corollary}\label{cor: two-sided BJ lower bound}
	Let $d=2$ and $0<r<R$. If $(x,\xi)\in\setR^4$ satisfies $\abs{x}\leq \frac{R-2r}3$ and
	\(|\xi|\leq \frac{1}{24R}\) then
	\[
		W_\BJ f_{r,R}(x,\xi) \gtrsim \log\oleft(\frac12 \frac{R-\abs{x}}{r+\abs{x}}\right)^2.
	\]
\end{corollary}

We can now give a direct proof of~\autoref{thm:attainment} in the case $d=2$.

\begin{proof}[Proof of \autoref{thm:attainment}, case $d=2$]
	The elementary building block of our construction is the annular profile $f_{r,R}$, which can be seen as a truncation of the scale-invariant function $\abs{x}^{-1}$ in dimension $2$. The Born--Jordan distribution of such profiles grows like $\log(R/r)^2$ near the origin, while in the same regime we have $\norm{f_{r,R}}_{L^2(\setR^2)}^2 \sim \log(R/r)$. This motivates forming a superposition like
	\[
		f = \sum_{n=1}^\infty a_n f_n, \qquad f_n = f_{r_n,R_n},
	\]
	so that the $n$-th annulus contributes at scale $r_n$ roughly like $a_n^2 \lambda_n^2$, where we set $\lambda_n=\log(R_n/r_n)$, while the total $L^2$ norm is controlled by $\sum_n a_n^2\lambda_n$. We will choose the radii such that the annuli are pairwise disjoint and shrink rapidly to the origin, allowing the $n$-th annulus to dominate $W_{\BJ}f$ on a shrinking phase-space box $B(0,r_n)\times B(0,r_n)$.

	A convenient, explicit way to implement this program is to choose $\lambda_n \coloneqq n^3+2$ and $a_n \coloneqq n^{-5/2}$, so that
	\[
		\sum_{n=1}^\infty a_n^2\lambda_n<\infty, \qquad a_n^2\lambda_n^2\to\infty.
	\]
	Moreover, set $R_1=\dfrac{1}{10}$ and define recursively $r_n \coloneqq R_n e^{-\lambda_n}$ and $R_{n+1} \coloneqq \dfrac{r_n}{8}\quad(n\ge1)$, so that
	\[
		5r_n<R_n,\qquad R_{n+1}<r_n,\qquad R_n\le \frac{1}{10}\qquad \forall n\ge1.
	\]
	In particular, the annuli $\{r_n\le \abs{x}\le R_n\}$ are pairwise disjoint and contained in $B(0,1/10)$, hence $\norm{f}_{L^2(\setR^2)}^2
		=
		2\pi \sum_{n=1}^\infty a_n^2 \lambda_n
		<
		\infty$.

	Let us now prove that $W_\BJ f$ is unbounded in every neighborhood of the origin. Fix $n\ge1$ and $(x,\xi)\in B(0, r_n)\times B(0, r_n)$. Since $\supp f\subseteq B(0,1/10)$ and $ r_n\le r_1\le 1/10$, the assumptions of \autoref{lem: cross-Re W_tau(x xi) lower bound} are satisfied by $f$ with $R_0=R_1=R=1/10$, hence
	\[
		\Re W_\tau(f,f)(x,\xi)\ge \frac12 W_\tau(f,f)(x,0)\qquad \forall \tau\in(0,1).
	\]
	Integrating in $\tau$, we obtain
	\[
		\abs{W_{\BJ}f(x,\xi)}\ge \Re W_{\BJ}f(x,\xi)\ge \frac12 W_{\BJ}f(x,0).
	\]
	At frequency $\xi=0$ all cross-terms are nonnegative since $f\ge0$, so by Tonelli's theorem we have
	\[
		W_{\BJ}f(x,0)
		=\sum_{m,k\ge1} a_ma_k\,W_{\BJ}(f_m,f_k)(x,0)
		\ge a_n^2 W_{\BJ}f_n(x,0).
	\]
	Next, since $\abs{x}\le r_n$ and $5r_n<R_n$, we have $\abs{x}\le \dfrac{R_n-2r_n}{3}$, so \autoref{cor: two-sided BJ lower bound} with $\xi=0$ implies
	\[
		W_{\BJ}f_n(x,0)\gtrsim
		\log\left(\frac12\frac{R_n-\abs{x}}{r_n+\abs{x}}\right)^2 \ge \log \left(\frac12\frac{R_n- r_n}{2r_n} \right)^2 = \log \left( \frac{e^{\lambda_n}-1}{4} \right)^2.
	\]
	Since $\lambda_n\to\infty$, it is clear that for $n$ sufficiently large (say, $n\ge N$ for a suitable $N \in \setN_+$) we have
	\[
		\log\left(\frac{e^{\lambda_n}-1}{4}\right)\ge \frac12\lambda_n,
	\]
	hence the previous arguments show that there is a constant $C>0$ such that
	\[
		\abs{W_{\BJ}f(x,\xi)}\ge C a_n^2\lambda_n^2
		\qquad
		\forall (x,\xi)\in B(0,r_n)\times B(0, r_n), \quad n\ge N.
	\]
	Since $a_n^2\lambda_n^2=n^{-5}(n^3+2)^2\to\infty$, this bound implies that $W_{\BJ}f$ is not essentially bounded in any neighborhood of the origin.

	To conclude, let $z_0=(x_0,\xi_0)$ be a Lebesgue density point of $\Omega$ and set $g=\pi(z_0)f$. Covariance yields $\abs{W_{\BJ}g(x,\xi)}=\abs{W_{\BJ}f(x-x_0,\xi-\xi_0)}$, hence
	\[
		\abs{W_{\BJ}g(x,\xi)}\ge C a_n^2\lambda_n^2
		\qquad
		\forall (x,\xi)\in E_n \coloneqq B(x_0, r_n)\times B(\xi_0, r_n), \quad n \ge N.
	\]
	Since $z_0$ is a density point, we have $\leb^4(E_n\cap\Omega)>0$ for all $n$ sufficiently large. Therefore $g\in L^2(\setR^2)$ satisfies $\norm{W_{\BJ}g}_{L^\infty(\Omega)}=\infty$, as claimed.
\end{proof}

\begin{remark}
	It is worth noting that the mechanism behind the critical case $d=2$ is inherently $L^\infty$ in nature, as it relies on forcing controlled pointwise blow-up. As such, it cannot be transplanted to critical finite $p_*(d)=2d/(d-2)$ when $d>2$, since in this case the problem is about genuine integrability --- in particular, the same annular blocks contribute only a bounded amount to the $L^{p_*}$ norm due to critical scale invariance.
	This of course is consistent with the boundedness result in \cite{Wang26}.
\end{remark}

\subsubsection{Explicit proof of attainment: supercritical case in dimension \texorpdfstring{$d>2$}{d > 2}}
\label{sec:attainment-supercritical}

The proof of \autoref{thm:attainment} in the case $d>2$ is based on a different approximation of the singular profile $\abs{x}^{-d/2}$: instead of the truncations of \autoref{def:f_rR} already used in \autoref{sec:attainment-critical},
for every parameter $\alpha\in(0,d/2)$, we define the radial function $F_\alpha\in L^2(\setR^d)$ given by
\[
	F_\alpha(x)
	= \abs{x}^{-\alpha} \bm1_{[0,1]}(\abs{x}),
\]
whose squared $L^2$ norm is
\[
	\norm{F_\alpha}_{L^2(\setR^d)}^2
	= \int_{B(0,1)} \abs{x}^{-2\alpha} \d x
	= \int_0^1 r^{-2\alpha} \omega_{d-1} r^{d-1} \d r
	= \frac{\omega_{d-1}}{d-2\alpha}.
\]

\begin{proof}[Proof of \autoref{thm:attainment}, case $d>2$]
	The real part of the $\tau$-Wigner distribution of $F_\alpha$ is
	\[
		\Re W_\tau F_\alpha(x,\xi)
		= \int_{\setR^d} \cos(2\pi\xi\cdot y)
		\abs{x+\tau y}^{-\alpha} \abs{x-(1-\tau)y}^{-\alpha}
		\bm1_{[0,1]}(\abs{x+\tau y}) \bm1_{[0,1]}(\abs{x-(1-\tau)y}) \d y.
	\]

	As in the proof of \autoref{lem:Re W_tau(x xi) lower bound} we have that the conditions
	\begin{align*}
		\tau\abs{y}-\abs{x}     & \leq \abs{x+\tau y} \leq 1    &
		                        & \text{and}                    &
		(1-\tau)\abs{y}-\abs{x} & \leq \abs{x-(1-\tau)y} \leq 1
	\end{align*}
	imply
	\[
		\abs{y} \leq \min\bigl\{\tau^{-1},(1-\tau)^{-1}\bigr\}(1+\abs{x}) \leq 2(1+\abs{x}).
	\]
	As such, when computing $\Re W_\tau F_\alpha(x,\xi)$ we can restrict the integral to the ball
	$B(0,2(1+\abs{x}))$ and use the estimate
	\[
		\cos(2\pi\xi\cdot y) \geq \cos\bigl(4\pi(1+\abs{x})\abs{\xi}\bigr)
	\]
	whenever $4\pi(1+\abs{x})\abs{\xi} \leq \pi$, leading to the lower bound
	\[
		\Re W_\tau F_\alpha(x,\xi)
		\geq \cos\bigl(4\pi(1+\abs{x})\abs{\xi}\bigr) W_\tau F_\alpha(x,0).
	\]
	We now look for sufficient conditions on $x$ and $\tau$ such that
	\[
		\abs{x-(1-\tau)y}\leq1 \implies \abs{x+\tau y}\leq1
		\qquad \forall y\in\setR^d.
	\]
	This means
	\[
		B\left(\frac{x}{1-\tau},\frac1{1-\tau}\right)
		\subseteq
		B\left(\frac{-x}{\tau},\frac1{\tau}\right),
	\]
	which is ensured by
	\[
		\abs*{\frac{x}{1-\tau}-\frac{-x}{\tau}} + \frac1{1-\tau} \leq \frac1{\tau}.
	\]
	A direct computation shows that this condition is equivalent to $\abs{x}+2\tau\leq1$. Under this assumption, we may also simplify
	\[
		\bm1_{[0,1]}(\abs{x+\tau y})\bm1_{[0,1]}(\abs{x-(1-\tau)y})
		=
		\bm1_{[0,1]}(\abs{x-(1-\tau)y}).
	\]

	Moreover, with the substitution
	\begin{align*}
		u        & = (1-\tau)y-x,             &
		\d u     & = (1-\tau)^d \d y,         &
		x+\tau y & = \frac{x+\tau u}{1-\tau},
	\end{align*}
	we get
	\[
		\begin{split}
			W_\tau F_\alpha(x,0)
			 & = \int_{\setR^d}
			\abs*{\frac{x+\tau u}{1-\tau}}^{-\alpha} \abs{u}^{-\alpha}
			\bm1_{[0,1]}(\abs{u}) (1-\tau)^{-d} \d u                                        \\
			 & = (1-\tau)^{\alpha-d}
			\int_{B(0,1)} \abs{x+\tau u}^{-\alpha} \abs{u}^{-\alpha} \d u                   \\
			 & \geq (1-\tau)^{\alpha-d}
			\int_{B(0,1)} (\abs{x}+\tau)^{-\alpha} \abs{u}^{-\alpha} \d u                   \\
			 & = \frac{\omega_{d-1}}{d-\alpha} (1-\tau)^{\alpha-d} (\abs{x}+\tau)^{-\alpha} \\
			 & \geq \frac{\omega_{d-1}}{d-\alpha} (\abs{x}+\tau)^{-\alpha},
		\end{split}
	\]
	where in the last step we used $\alpha<d$ and $0<\tau<1$.

	With regard to Born--Jordan, we restrict the integral to $\tau\leq(1-\abs{x})/2$ in order to use the
	previous lower bound for $W_\tau F_\alpha(x,0)$ and obtain
	\[
		\begin{split}
			W_\BJ F_\alpha(x,0)
			 & = 2\int_0^{1/2} W_\tau F_\alpha(x,0) \d\tau
			\geq 2\int_0^{\frac{1-\abs{x}}2} W_\tau F_\alpha(x,0) \d\tau                  \\
			 & \geq \frac{2\omega_{d-1}}{d-\alpha}
			\int_0^{\frac{1-\abs{x}}2} (\abs{x}+\tau)^{-\alpha} \d\tau                    \\
			 & = \frac{2\omega_{d-1}}{d-\alpha}
			\left[\frac{(\abs{x}+\tau)^{1-\alpha}}{1-\alpha}\right]_0^{\frac{1-\abs{x}}2} \\
			 & = \frac{2\omega_{d-1}}{(d-\alpha)(\alpha-1)}
			\left[\abs{x}^{1-\alpha}-\left(\frac{1+\abs{x}}2\right)^{1-\alpha}\right].
		\end{split}
	\]

	In particular, if $\alpha>1$ and $\abs{x}$ is small, then $W_\BJ F_\alpha(x,0)\gtrsim \abs{x}^{1-\alpha}$. To conclude, we need to upgrade this lower bound from the slice $\xi=0$ to a full phase-space neighborhood of the origin. To this aim, choose $\rho,\sigma>0$ so small that
	\[
		4\pi(1+\abs{x})\abs{\xi}\le \frac{\pi}{3}
		\qquad
		\forall \abs{x}\le \rho,\  \abs{\xi}\le \sigma.
	\]
	Then the previous estimate for $\Re W_\tau F_\alpha(x,\xi)$ yields
	\[
		\Re W_\tau F_\alpha(x,\xi)\ge \frac12 W_\tau F_\alpha(x,0)
		\qquad
		\forall \abs{x}\le \rho,\  \abs{\xi}\le \sigma,\ \tau\in(0,1).
	\]
	Since $W_\BJ F_\alpha$ is real-valued for $d>1$, integrating in $\tau$ gives
	\[
		\begin{split}
			W_\BJ F_\alpha(x,\xi)
			 & = \Re W_\BJ F_\alpha(x,\xi)
			= \int_0^1 \Re W_\tau F_\alpha(x,\xi)\,\d\tau        \\
			 & \ge \frac12 \int_0^1 W_\tau F_\alpha(x,0)\,\d\tau
			= \frac12 W_\BJ F_\alpha(x,0)
		\end{split}
	\]
	for all $\abs{x}\le \rho$ and $\abs{\xi}\le \sigma$. Consequently, we infer
	\[
		W_\BJ F_\alpha(x,\xi)\gtrsim \abs{x}^{1-\alpha}
		\qquad
		\forall x\in B(0,\rho)\setminus\{0\},\ \xi\in B(0,\sigma).
	\]

	In the case where  $p\in\bigl(p_*(d),\infty\bigr)$ it is enough (and possible) to choose $\alpha$ so that
	\[
		1+\frac{d}{p}<\alpha<\frac d2.
	\]
	Then $p(1-\alpha)<-d$, so $\abs{x}^{1-\alpha}\notin L^p_{\mathrm{loc}}(\setR^d)$, and therefore $W_\BJ F_\alpha\notin L^p_{\mathrm{loc}}(\setR^{2d})$, as claimed.

	If $p=\infty$, choose any
	\[
		1<\alpha<\frac d2,
	\]
	which is possible because $d>2$. Then $\abs{x}^{1-\alpha}\to\infty$ as $x\to0$, and the previous bound shows that $W_\BJ F_\alpha$ is unbounded on every neighborhood of the origin. As a consequence, $W_\BJ F_\alpha\notin L^\infty_{\mathrm{loc}}(\setR^{2d})$.

	This yields a completely explicit local counterexample for every $p\in\bigl(p_*(d),\infty\bigr]$.
	The full statement of \autoref{thm:attainment} on an arbitrary measurable set $\Omega$ follows
	from the usual covariance argument, after picking $(x_0,\xi_0)$ as in \eqref{eq:Lebesgue differentiation x}
	(with the same $\sigma$ as above) and using the bound $\abs{W_\BJ(\pi(x_0,\xi_0)F_\alpha)(x,\xi)}\gtrsim\abs{x-x_0}^{1-\alpha}$
	on $B(x_0,r)\times B(\xi_0,\sigma)$ as $r\to0$.
\end{proof}

\begin{appendices}
	\section{Incomplete Born--Jordan distributions}\label{sec:incomplete}

	It is intuitively clear at this point that the endpoint singularities at $\tau=0$ and $\tau=1$ are the ultimate source of the pathologies of the full Born--Jordan distribution in dimension $d\ge 2$. It is therefore natural to isolate and study suitably truncated averages:
	\begin{definition}
		Given $\delta\in[0,1/2]$, the \emph{incomplete Born--Jordan distribution} of $f \in L^2(\setR^d)$ is
		\begin{equation}\label{eq:BJ-delta}
			W_\BJ^\delta f(z) \coloneqq \int_\delta^{1-\delta} W_\tau f(z) \d\tau.
		\end{equation}
	\end{definition}

	It is clear that $W_\BJ f = W_\BJ^0 f$. More interestingly, a local incompleteness property for annular supports holds: we show below that if $f$ is supported on an annulus, then $W_\BJ f=W_{\BJ}^\delta f$ near the origin for some $\delta$. This property, already used implicitly in the main text, comes with additional benefits. For instance, it yields the local continuity of $W_\BJ f$ near the origin for arbitrary $f\in L^2(\setR^d)$ supported in an annulus. In turn, this regularity can be leveraged to provide an alternative proof of the unboundedness of the concentration problem~\eqref{eq:concentration problem} for $p=\infty$.

	We start by giving the precise statement of the incompleteness property.

	\begin{lemma}\label{lem:W-tau-support}
		Let $f\in L^2(\setR^d)$ be such that
		\[
			\supp f \subseteq \set{x\in\setR^d}{r_1\leq\abs{x}\leq r_2}
		\]
		for some radii $0<r_1<r_2$.
		If $\delta\in[0,1/2]$ and $\rho>0$ satisfy
		\begin{equation}\label{eq:annulus-condition}
			(r_1+r_2)\delta + \rho \leq r_1
		\end{equation}
		then
		\begin{align}
			W_\tau f(z) & = 0                                                                                               &
			            & \forall z = (x,\xi) \in B(0,\rho)\times\setR^d, \quad \forall \tau\in (0,\delta]\cup[1-\delta,1),
		\end{align}
		and, as a consequence, we have
		\begin{align}
			W_\BJ f(z) & = W_\BJ^\delta f(z)
			           & \forall z = (x,\xi) \in B(0,\rho)\times\setR^d.
		\end{align}
	\end{lemma}

	\begin{proof}
		Suppose that $W_\tau f(x,\xi)\neq0$ for some $x\in B(0,\rho)$.
		Then there must exist $y\in\setR^d$ such that both factors $f\bigl(x+\tau y\bigr)$ and
		$f\bigl(x-(1-\tau)y\bigr)$ are nonzero, which requires
		\begin{align*}
			r_1 & \leq \abs{x+\tau y} \leq r_2     &
			    & \text{and}                       &
			r_1 & \leq \abs{x-(1-\tau)y} \leq r_2.
		\end{align*}
		Applying the triangle inequality to the first and last inequalities leads to
		\begin{gather*}
			r_1 \leq \abs{x+\tau y} \leq \abs{x} + \tau\abs{y} < \rho + \tau\abs{y}, \\
			-\rho + (1-\tau)\abs{y} < -\abs{x} + (1-\tau)\abs{y} \leq \abs{x-(1-\tau)y} \leq r_2.
		\end{gather*}
		Chaining the first inequality multiplied by $1-\tau$ and the second by $\tau$ results in%
		\footnote{At least one of the two strict inequalities survives because $\tau$ and $1-\tau$
			cannot be both zero.}
		\[
			(1-\tau)(r_1-\rho) < \tau(r_2+\rho),
		\]
		which is equivalent to
		\[
			r_1 < (r_1+r_2)\tau + \rho
		\]
		and together with \eqref{eq:annulus-condition} implies $\tau>\delta$. The fact that
		$\tau$ must also be smaller than $1-\delta$ follows from a similar argument involving
		the second and third inequalities, or from the identity $W_{1-\tau}f(z)=\overline{W_\tau f(z)}$.
	\end{proof}

	We are now ready to prove that incomplete Born--Jordan averages are globally regular.
	\begin{proposition}
		\label{prop:BJ-delta-continuity}
		Let $\delta\in(0,1/2]$ and $f\in L^2(\setR^d)$. Then $W_\BJ^\delta f$ is continuous.
	\end{proposition}

	\begin{proof}
		By \cite[Proposition~3.9]{CT20}, the \(\tau\)-Wigner distributions are continuous functions of \(z\) for any \(\tau\in (0,1)\) and any \(f\in L^2(\setR^d)\).
		As such, for any fixed \(\tau\) we have \(\lim_{z\to z_0} W_{\tau}f(z)=W_\tau f(z_0)\).
		By the pointwise estimate \cite[Proposition~6.4]{Boggiatto-deDonno-Oliaro-2010}
		\[
			|W_\tau f(z)|\leq \|W_\tau f\|_{L^\infty(\setR^{2d})}\leq \frac{1}{(\tau(1-\tau))^{d/2}}\|f\|_{L^2(\setR^d)}^2
		\]
		and the fact that \(\frac{1}{(\tau(1-\tau))^{d/2}}\) is integrable on \([\delta,1-\delta]\) for any
		\(\delta\in (0,1/2]\) the claim follows by dominated convergence.
	\end{proof}

	The incompleteness property stated in \autoref{lem:W-tau-support}, combined with the continuity of
	$W_\BJ^\delta$ stated in \autoref{prop:BJ-delta-continuity}, leads directly to the continuity of
	$W_\BJ$ near the origin.

	\begin{corollary}
		Let $f\in L^2(\setR^d)$ be such that $\supp f \subseteq \set{x\in\setR^d}{r_1\leq\abs{x}\leq r_2}$
		for some radii $0<r_1<r_2$. If $\delta\in (0,1/2]$ and $\rho>0$ satisfy \eqref{eq:annulus-condition}
		then $W_\BJ f$ is continuous in $B(0,\rho)\times\setR^d$.
	\end{corollary}

	The continuity of $W_\BJ$ near the origin can be exploited to provide an alternative solution to the
	concentration problem~\eqref{eq:concentration problem} in the case $p=\infty$. The reason is that,
	thanks to the continuity, it is sufficient to show that $\abs{W_\BJ f(0,0)}$ diverges sufficiently
	fast, instead of proving that the estimates of \autoref{cor:W_BJ lower bound} hold uniformly in the
	product of suitable balls.

	\begin{proof}
		[Alternative proof of \autoref{thm:critical d=2}, and \autoref{thm:supercritical d>2} for $p=\infty$]
		We will consider the functions $f_R$ as before.
		Since $\supp f_R\subseteq\set{x\in\setR^d}{1\leq\abs{x}\leq R}$, we can apply \autoref{lem:W-tau-support}
		with $r_1=1,r_2=R,\delta=\frac1{2(R+1)},\rho=1/2$ and obtain that $W_\BJ f_R=W_\BJ^\delta f_R$
		in $B(0,\rho)\times\setR^d$; applying then \autoref{prop:BJ-delta-continuity} we deduce that
		$W_\BJ f_R$ is continuous in this very open set. By continuity and covariance, unboundedness will now follow if we can show that \(W_\BJ f_R(0,0)\) grows faster than \(\|f_R\|_{L^2(\setR^d)}^2=\omega_{d-1}\log(R)\) as \(R\to \infty.\) We will prove this by direct computation.
		Let \(\tau \in (0,1/2]\). By definition
		\[
			\begin{split}
				W_\tau f_R(0,0)
				 & = \int_{\setR^d} f_R\bigl(\tau y\bigr) \overline{f_R\bigl(-(1-\tau)y\bigr)} \d y \\
				 & = \int_{\setR^d} \abs{\tau y}^{-d/2} \bm1_{[1,R]}(\abs{\tau y})
				\abs{(1-\tau)y}^{-d/2} \bm1_{[1,R]}\bigl(\abs{(1-\tau)y}\bigr) \d y.
			\end{split}
		\]
		As usual, the indicator functions are both non-zero when
		$\frac1\tau\leq \abs{y} \leq \frac{R}\tau$ and  $\frac1{1-\tau} \leq \abs{y} \leq \frac{R}{1-\tau}$,
		which is equivalent to
		\[
			\frac1\tau \leq \abs{y} \leq \frac{R}{1-\tau}.
		\]
		Therefore
		\[
			\begin{split}
				W_\tau f_R(0,0)
				 & = \bigl(\tau(1-\tau)\bigr)^{-d/2}
				\int_{\setR^d} \bm1_{\left[\frac1\tau, \frac{R}{1-\tau}\right]}(\abs{y}) \abs{y}^{-d} \d y                       \\
				 & =\begin{cases}
					    \omega_{d-1} \bigl(\tau(1-\tau)\bigr)^{-d/2}\log(R\tau/(1-\tau)) & (\tau >\frac{1}{R+1})      \\
					    0                                                                & (\tau \leq \frac{1}{R+1}).
				    \end{cases}
			\end{split}
		\]
		We may now compute the Born--Jordan distribution at \(0\), but we must consider \(d=2\) and \(d>2\) separately. For \(d=2\) we have, using the change of variables \(u=\tau/(1-\tau)\)
		\[
			W_{\BJ}f_R(0,0)
			= 2\omega_{1}\int_{1/(R+1)}^{1/2}\frac{\log(R\tau/(1-\tau))}{\bigl(\tau(1-\tau)\bigr)}\d\tau
			= 4\pi \int_{1/R}^{1}\frac{\log(Ru)}{u}\d u
			= 2\pi (\log R)^2.
		\]
		The unboundedness of the supremum clearly follows. For \(d>2\) we use the same change of variables, plus integration by parts:
		\[
			\begin{split}
				W_{\BJ}f_R(0,0)
				 & = 2\omega_{d-1}\int_{1/(R+1)}^{1/2}\frac{\log(R\tau/(1-\tau))}{\bigl(\tau(1-\tau)\bigr)^{d/2}}\d\tau
				= 2\omega_{d-1} \int_{1/R}^{1}\frac{\log(Ru)}{u^{d/2}}(u+1)^{d-2}\d u                                   \\
				 & \geq 2\omega_{d-1} \int_{1/R}^{1}\frac{\log(Ru)}{u^{d/2}}\d u                                        \\
				 & =2\omega_{d-1}\left[\frac{\log(Ru)}{(1-d/2)u^{d/2-1}}-\frac{1}{(1-d/2)^2u^{d/2-1}}\right]_{1/R}^1    \\
				 & =2\omega_{d-1}\left(\frac{\log(R)}{1-d/2}-\frac{1}{(d/2-1)^2}+\frac{R^{d/2-1}}{(d/2-1)^2}\right)     \\
				 & \gtrsim R^{d/2-1}. \qedhere
			\end{split}
		\]
	\end{proof}

	To further support the perspective claimed at the beginning of this appendix, we record that the $L^p$ concentration problem for the incomplete Born--Jordan distributions $W_{\BJ}^\delta$ with $0<\delta\le1/2$ always has subcritical nature: the concentration functional $f \mapsto \norm{W_{\BJ}^\delta f}_{L^p(\Omega)}$ is sequentially weakly continuous for every $1 \le p < \infty$ and every measurable set $\Omega$ with $\leb^{2d}(\Omega)<\infty$. In fact, the same is true even for gentler truncations, for instance, after replacing $\bm1_{[\delta,1-\delta]}$ with a suitable weight $w \in L^1(0,1)$ in \eqref{eq:BJ-delta}, namely satisfying $\int_0^1\abs{w(\tau)}\bigl(\tau(1-\tau)\bigr)^{-d(1/2-1/p)}\d\tau<\infty$ --- see \cite{cuong} in this connection. Moreover, one can easily obtain sharp asymptotics for the optimal concentration value as $\delta \downarrow 0$ in dimension $d=2$ and critical regime $p_*(2)=\infty$.

	\begin{proposition}\label{prop:d2-delta-asymptotic}
		Let $d=2$ and let $\Omega\subseteq\setR^2\times\setR^2$ be measurable with
		$0<\leb^4(\Omega)<\infty$. For $\delta\in(0,1/2]$, define
		\[
			L_{\infty,\Omega}^\delta
			\coloneqq
			\sup_{\norm{f}_{L^2(\setR^2)}=1}
			\norm{W_\BJ^\delta f}_{L^\infty(\Omega)}.
		\]
		Then $L_{\infty,\Omega}^\delta \asymp \log(1/\delta)$ as $\delta \downarrow 0$.
	\end{proposition}

	\begin{proof}
		The pointwise estimate for $W_\tau$ yields
		\[
			\norm{W_\BJ^\delta f}_{L^\infty(\Omega)}
			\le \norm{f}_{L^2}^2 \int_\delta^{1-\delta} (\tau(1-\tau))^{-1}\,\d\tau	\lesssim \log(1/\delta)\norm{f}_{L^2}^2,
		\]
		hence $L_{\infty,\Omega}^\delta\lesssim \log(1/\delta)$.
		For the companion lower bound, choose $R_\delta = \dfrac{1}{\delta}-2$ with $\delta \le 1/3$ to ensure $R_\delta\ge 1$.
		Then $(1+R_\delta)\delta = 1-\delta$, so by \autoref{lem:W-tau-support} with $\rho=\delta/2>0$ we have $W_\BJ^\delta f_{R_\delta}(0,0)=W_\BJ f_{R_\delta}(0,0)$.
		Arguing as in the previous alternative proof gives
		\[
			W_\BJ f_{R_\delta}(0,0)=2\pi(\log R_\delta)^2,
			\qquad
			\norm{f_{R_\delta}}_{L^2(\setR^2)}^2=2\pi (\log R_\delta),
		\]
		therefore
		\[
			\frac{W_\BJ^\delta f_{R_\delta}(0,0)}{\norm{f_{R_\delta}}_{L^2}^2}
			= \log R_\delta \asymp \log(1/\delta).
		\]
		Since $W_\BJ^\delta$ is continuous by \autoref{prop:BJ-delta-continuity}, if $z_0$ is a Lebesgue density point of $\Omega$ then the shifted function
		$\pi(z_0)f_{R_\delta}$ satisfies
		\[
			\norm{W_\BJ^\delta(\pi(z_0)f_{R_\delta})}_{L^\infty(\Omega)}
			\gtrsim \log(1/\delta)\norm{f_{R_\delta}}_{L^2}^2,
		\]
		which proves the matching lower bound for $L_{\infty,\Omega}^\delta$.
	\end{proof}

	\begin{remark}
		Similar arguments leveraging \autoref{cor:W_BJ lower bound} show that, again for $d=2$,
		the Born--Jordan concentration value $L_{p,\Omega}\coloneqq \sup_{\norm{f}_{L^2}=1}\norm{W_\BJ f}_{L^p(\Omega)}$
		satisfies $L_{p,\Omega} \asymp p$ as $p \uparrow \infty$. Indeed, the upper bound follows from
		the estimate $\norm{W_\tau f}_{L^p(\Omega)}\leq(\tau(1-\tau))^{-(1-2/p)}\norm{f}_{L^2}^2$ used
		in the proof of \autoref{thm:subcritical}, since $\int_0^1(\tau(1-\tau))^{-(1-2/p)}\d\tau\sim p$
		as $p\to\infty$. The lower bound follows by testing $f_R$ with $R=e^{p/2}$ in \autoref{cor:W_BJ lower bound},
		on a product set of measure comparable to $R^{-2}$ as in the proof of \autoref{thm:supercritical d>2},
		leading to a quotient of order $(\log R)\,R^{-2/p}\asymp p$ --- the implicit constants depend on $\Omega$ in general.
	\end{remark}

	\section{Born--Jordan quantization of local Lebesgue symbols}\label{sec:quantization}

	Let us briefly examine some consequences of our main results on Born--Jordan
	pseudodifferential operators. Recall the standard weak correspondence between time-frequency representations and quantization rules: the
	Born--Jordan operator with symbol \(a \in \mathcal{S}'(\setR^{2d})\) is defined by
	\[
		\angles[\big]{\Op_\BJ(a)f,g} = \angles[\big]{a,W_\BJ(g,f)}, \qquad \forall\ f,g \in \mathcal{S}(\setR^{d}).
	\]
	Since
	\[
		W_\BJ(g,f)=\int_0^1 W_\tau(g,f) \d\tau ,
	\]
	this is precisely the operator \(T_Q^a\) studied in \cite[Section~6]{Boggiatto-deDonno-Oliaro-2010}.
	Born--Jordan operators are a priori only well-defined operators from \(\mathcal{S}(\setR^d)\) to \(\mathcal{S}'(\setR^d)\), but the subcritical concentration scenario immediately yields $L^2$-boundedness results for such operators for symbols with local Lebesgue summability.
	To be precise, let
	\(\Omega\subset\setR^{2d}\) be measurable with \(0<\leb^{2d}(\Omega)<\infty\), let
	\(1\le p<p_*(d)\) and set \(q=p'\). If \(a\in L^q(\Omega)\) and \(\supp a\subseteq\Omega\), then
	\[
		\begin{split}
			\abs{\angles[\big]{\Op_\BJ(a)f,g}}
			 & = \abs{\angles[\big]{a,W_\BJ(g,f)}}                                                    \\
			 & \le \norm{a}_{L^q(\Omega)} \norm{W_\BJ(g,f)}_{L^p(\Omega)}                             \\
			 & \le C_{p,\Omega}\norm{a}_{L^q(\Omega)} \norm{f}_{L^2(\setR^d)}\norm{g}_{L^2(\setR^d)},
		\end{split}
	\] for some $C_{p,\Omega}>0$. As a consequence, \(\Op_\BJ(a)\) extends to a bounded operator on \(L^2(\setR^d)\) with
	\[
		\norm{\Op_\BJ(a)}_{L^2 \to L^2} \le C_{p,\Omega}\norm{a}_{L^q(\Omega)}, \qquad q> \begin{cases} \dfrac{2d}{d+2} & (d>2) \\ 1 & (d=2). \end{cases}
	\]
	This result is already contained in \cite[Theorem~6.10]{Boggiatto-deDonno-Oliaro-2010}. Indeed, in the case \(p=2\) this reads
	\[
		a\in L^q(\setR^{2d}) \longmapsto \Op_\BJ(a)\in \mathcal{B}(L^2(\setR^d)), \qquad \frac{2d}{d+2}<q\le 2.
	\]
	For symbols supported in a finite-measure set, the case \(q>2\) also follows from their \(q=2\) result and the embedding \(L^q(\Omega)\subset L^2(\Omega)\).

	On the other hand, these sufficient conditions seem to lack a companion negative counterpart. We will now show that the complementary failure below the critical threshold actually follows from our unboundedness results. Assume indeed \(q<2d/(d+2)\), so that \(q'>p_*(d)\). If a uniform estimate
	\[
		\abs{\angles[\big]{\Op_\BJ(a)f,g}} \le C\norm{a}_{L^q(\Omega)} \norm{f}_{L^2(\setR^d)} \norm{g}_{L^2(\setR^d)}
	\]
	held for all \(a\in L^q(\Omega)\), then taking \(g=f\) would imply by duality
	\[
		\norm{W_\BJ f}_{L^{q'}(\Omega)} \le C\norm{f}_{L^2(\setR^d)}^2 \qquad \forall\ f\in L^2(\setR^d).
	\]
	Such a claim clearly contrasts the supercritical counterexamples of \autoref{thm:supercritical d>2} and
	\autoref{thm:attainment-baire}, hence no such local \(L^q\)-symbol estimate can hold below the threshold. In dimension \(d=2\), the endpoint \(q=1\) is similarly ruled out. Indeed, \(q'=\infty\), and
	\autoref{thm:critical d=2} and \autoref{thm:attainment-baire} give functions \(f\in L^2(\setR^2)\) with \(W_\BJ f\notin L^\infty(\Omega)\).

	The endpoint case $q=2d/(d+2)$ can be settled as a consequence of~\cite[Corollary 1.2]{Wang26}. Indeed, for $d\geq3$ set $p_*=p_*(d)$ and $q_*=p_*'=2d/(d+2)$; since $\overline{W_\BJ(g,f)}=W_\BJ(f,g)$ we have $\angles{\Op_\BJ(a)f,g}=\int_{\setR^{2d}}a\,W_\BJ(f,g)$, hence H\"older's inequality and \cite[Theorem 1.1]{Wang26} give $L^2$ boundedness:
	\[ \norm{\Op_\BJ(a)}_{L^2\to L^2}\leq C_d\norm{a}_{L^{q_*}},\qquad a\in L^{q_*}(\setR^{2d}). \]
	In fact, a stronger conclusion holds at this endpoint, namely that the Born--Jordan operators are compact. We denote below by $\mathfrak S_r$ the Schatten class of order $r$ on $L^2(\setR^d)$.

	\begin{proposition}\label{prop:bj-endpoint-compact}
		For every $d\geq3$ and $a\in L^{q_*}(\setR^{2d})$ the operator $\Op_\BJ(a)$ is compact on $L^2(\setR^d)$. In particular, if $\leb^{2d}(\Omega)<\infty$ and $q_*\leq q\leq\infty$, every $a\in L^q(\Omega)$ (extended by zero outside $\Omega$) defines a compact Born--Jordan operator.
	\end{proposition}

	\begin{proof}
		Consider $b\in L^2(\setR^{2d})$ and recall that the $\tau$-quantization $\Op_\tau(b)$, defined for $0\leq\tau\leq1$ by $\angles{\Op_\tau(b)f,g}=\angles{b,W_\tau(g,f)}$, has the integral kernel $k_\tau(x,y)=\mathcal F_2^{-1}b((1-\tau)x+\tau y,x-y)$, where $\mathcal F_2^{-1}$ denotes the inverse Fourier transform in the second variable. Plancherel's theorem and the (unimodular) change of variables $(x,y)\mapsto((1-\tau)x+\tau y,x-y)$ give $\norm{\Op_\tau(b)}_{\mathfrak S_2}=\norm{b}_{L^2}$. Since $\tau\mapsto k_\tau$ is continuous as a mapping with values in $L^2(\setR^{2d})$, the Born--Jordan operator is given by the $\mathfrak S_2$-valued average $\Op_\BJ(b)=\int_0^1\Op_\tau(b)\d\tau$, hence $\norm{\Op_\BJ(b)}_{\mathfrak S_2}\leq\norm{b}_{L^2}$ and $\Op_\BJ(b)$ is compact.

		Finally, given $a\in L^{q_*}(\setR^{2d})$ we choose $a_n\in L^2\cap L^{q_*}$ with $a_n\to a$ in $L^{q_*}$; by the endpoint bound we then have $\Op_\BJ(a_n)\to\Op_\BJ(a)$ in operator norm, hence $\Op_\BJ(a)$ is compact as claimed. The last claim follows from the embedding $L^q(\Omega)\subseteq L^{q_*}(\Omega)$.
	\end{proof}

	It is worth noting that interpolating the bounds $L^{q_*}\to\mathcal K$ (where $\mathcal K$ is the space of compact operators on $L^2(\setR^d)$ with the operator norm) and $L^2\to\mathfrak S_2$ obtained above also gives, for $q_*<q\leq2$,
	\[ \norm{\Op_\BJ(a)}_{\mathfrak S_r}\leq C_d^{1-\theta}\norm{a}_{L^q}, \qquad \theta=\frac{d+2}{2}-\frac dq,  \qquad r=\frac{2}{\theta} = \frac{4q}{(d+2)q-2d}.\]
	Moreover, compactness has the following consequence for Born--Jordan distributions.

	\begin{corollary}\label{cor:bj-critical-weak-pairings}
		For $d\geq3$, if $f_n\weakto f$ and $g_n\weakto g$ in $L^2(\setR^d)$, then $W_\BJ(f_n,g_n)\weakto W_\BJ(f,g)$ weakly in $L^{p_*}(\setR^{2d})$. As a consequence, for every measurable $\Omega\subseteq\setR^{2d}$ the critical functional $J_\BJ^{p_*}$ is sequentially weakly lower semicontinuous on $L^2(\setR^d)$.
	\end{corollary}

	\begin{proof}
		Given $a\in L^{q_*}(\setR^{2d})$, \autoref{prop:bj-endpoint-compact} implies that $\Op_\BJ(a)f_n\to\Op_\BJ(a)f$ in norm, while $g_n\weakto g$. As a result, we have
		\[ \int_{\setR^{2d}} a\,W_\BJ(f_n,g_n)=\angles{\Op_\BJ(a)f_n,g_n}\to\angles{\Op_\BJ(a)f,g}= \int_{\setR^{2d}} a\,W_\BJ(f,g), \]
		which is the claimed weak convergence since $L^{q_*}=(L^{p_*})^*$. Restricting to $\Omega$ and using the weak lower semicontinuity of the norm gives the second part of the claim.
	\end{proof}

	Combined with \autoref{thm:critical d>2}, \autoref{cor:bj-critical-weak-pairings} shows that the critical functional is sequentially weakly lower but not upper semicontinuous: along the sequence $g_n\weakto0$ of \autoref{thm:critical d>2}, the distributions $W_\BJ g_n$ converge weakly to zero in $L^{p_*}(\setR^{2d})$, while $\norm{W_\BJ g_n}_{L^{p_*}(\Omega)}$ stays bounded away from zero. The failure of norm convergence for such a sequence is thus witnessed by concentration on shrinking neighborhoods of a single phase-space point, which cannot be detected by pairing against any fixed symbol in $L^{q_*}$.

\end{appendices}

\bigskip
\section*{Acknowledgments}

The development of this work started on the occasion of a visit of the second author to the others in January 2026. The authors used large language models to assist with reviewing mathematical arguments and drafting the manuscript. The authors edited the outputs as needed and independently validated all mathematical arguments; they take full responsibility for the content of the paper.

\clearpage
\printbibliography[heading=bibintoc]

\end{document}